\documentclass[10pt,reqno]{amsart}
\usepackage{amssymb,mathrsfs,graphicx}
\usepackage{ifthen}
\usepackage{colortbl}
\definecolor{black}{rgb}{0.0, 0.0, 0.0}
\definecolor{red}{rgb}{1.0, 0.5, 0.5}
\provideboolean{shownotes} 
\setboolean{shownotes}{true} 
\newcommand{\margnote}[1]{
\ifthenelse{\boolean{shownotes}}%
{\marginpar{\raggedright\tiny\texttt{#1}}}%
{}%
}
\newcommand{\hole}[1]{
\ifthenelse{\boolean{shownotes}}%
{\begin{center} \fbox{ \rule {.25cm}{0cm} \rule[-.1cm]{0cm}{.4cm}
\parbox{.85\textwidth}{\begin{center} \texttt{#1}\end{center}} \rule
{.25cm}{0cm}}\end{center}} {} }

\title[Vlasov-Fokker-Planck equation with local alignment forces]{Global classical solutions of the Vlasov-Fokker-Planck equation with local alignment forces}

\author[Choi]{Young-Pil Choi}
\address[Young-Pil Choi]{\newline Fakult\"at f\"ur Mathematik \newline
Technische Universit\"at M\"unchen, Boltzmannstra\ss e 3, 85748, Garching bei M\"unchen, Germany}
\email{ychoi@ma.tum.de}

\numberwithin{equation}{section}

\newtheorem{theorem}{Theorem}[section]
\newtheorem{lemma}{Lemma}[section]

\newtheorem{proposition}{Proposition}[section]
\newtheorem{remark}{Remark}[section]

\newcommand{\R}{\mathbb R}

\newcommand{\ls}{\lesssim}

\newcommand{\T}{\mathbb T}
\newcommand{\N}{\mathbb N}

\newcommand{\mc}{\mathcal C}

\newcommand{\bb} {\mathbf B}

\newcommand{\bq}{\begin{equation}}
\newcommand{\eq}{\end{equation}}
\newcommand{\e}{\varepsilon}
\newcommand{\lt}{\left}
\newcommand{\rt}{\right}
\newcommand{\lal}{\langle}
\newcommand{\ral}{\rangle}
\newcommand{\pa}{\partial}

\newcommand{\bl}{\mathbf{L}}
\newcommand{\bpp}{\mathbf{P}}
\newcommand{\bi}{\mathbf{I}}
\newcommand{\mi}{\mathrm{i}}
\newcommand{\wh}{\widehat}
\newcommand{\oper}{\{\bi - \bpp\}}

\def\charf {\mbox{{\text 1}\kern-.30em {\text l}}}

\begin{document}
\allowdisplaybreaks

\date{\today}



\begin{abstract}In this paper, we are concerned with the global well-posedness and time-asymptotic decay of the Vlasov-Fokker-Planck equation with local alignment forces. The equation can be formally derived from an agent-based model for self-organized dynamics which is called Motsch-Tadmor model with noises. We present the global existence and uniqueness of classical solutions to the equation around the global Maxwellian in the whole space. For the large-time behavior, we show the algebraic decay rate of solutions towards the equilibrium under suitable assumptions on the initial data. We also remark that the rate of convergence is exponential when the spatial domain is periodic. The main methods used in this paper are the classical energy estimates combined with hyperbolic-parabolic dissipation arguments.
\end{abstract}

\maketitle \centerline{\date}


%
%
%
%
\section{Introduction}
We are concerned with a kinetic flocking equation in the presence of diffusion. More precisely, let $F = F(x,v,t) \geq 0$ be the density of particles which have position $x \in \R^d$ and velocity $v \in \R^d$ at time $t \geq 0$. Then the evolution of the density $F$ is governed by 
\bq\label{eq_1}
\pa_t F + v \cdot \nabla_x F + \nabla_v \cdot ((u_F - v)F) = \Delta_v F,
\eq
with the initial data 
\bq\label{ini_eq_1}
F(x,v,0) =: F_0(x,v),
\eq
where $u_F$ is the averaged local velocity given by
\[
u_F(x,t) := \frac{\int_{\R^d} vF(x,v,t)\,dv}{\int_{\R^d} F(x,v,t)\,dv}.
\]
The purpose of this paper is to study the global existence and uniqueness of classical solutions to the equation \eqref{eq_1} near the global Maxwellian and the large-time behavior of solutions.

Recently, Motsch and Tadmor introduced in \cite{MT} a new model for self-organized dynamics in which the alignment force is normalized with a local average density. Compared to the celebrated Cucker-Smale flocking model \cite{CS} where the alignment is scaled with the total mass, the Motsch-Tadmor model takes into account not only the distance between agents but also their relative distance. More specifically, these two prototype models are described by
\begin{equation}\label{level-par}
\frac{dx_i}{dt} = v_i, \quad \frac{dv_i}{dt} = \frac{1}{S_i(x)}\sum_{j=1}^N\psi(x_i - x_j)(v_j - v_i), \quad t > 0, \quad i \in \{1,\dots,N\},
\end{equation}
where $x_i(t) = (x_i^1(t),\cdots,x_i^d(t)) \in \R^d$ and $v_i(t) = (v_i^1(t),\cdots,v_i^d(t)) \in \R^d$ are the position and velocity of $i$-th agent at time $t$, respectively. 
Here, the scaling function $S_i(x)$ and the communication weight $\psi (x)$ are defined by
\begin{displaymath}
S_i(x) := \left\{ \begin{array}{ll}
N & \mbox{for Cucker-Smale model},\\
\displaystyle \sum_{k=1}^N \psi(x_i - x_k) & \mbox{for Motsch-Tadmor model}
\end{array}\right. 
\end{displaymath}
and $\psi(x) = 1/(1+|x|^2)^{\beta/2}$ with $\beta >0$, respectively. We refer to \cite{CFRT, CHL, HT, MT} for the existence and large-time behavior of solutions to these two models.
Note that the motion of an agent governed by Cucker-Smale model is modified by the total number of agents. Thus it can not be adopted to describe for far-from-equilibrium flocking dynamics. On the other hand, the Motsch-Tadmor model does not involve any explicit dependence on the number of agents, and this remedies several drawbacks of Cucker-Smale model outlined above. For a detailed description of the modeling and related literature, we refer readers to \cite{MT} and the references therein. 

When the number of agents governed by Motsch-Tadmor model goes to infinity, i.e., $N \to \infty$, we can formally derive a mesoscopic description for the system \eqref{level-par} with a density function $F=F(x,v,t)$ which is a solution to the following Vlasov-type equation:
\bq\label{kinetic-mt}
\left\{ \begin{array}{ll}
\partial_t F + v \cdot \nabla_x F + \nabla_v \cdot ((\tilde u_F - v)F) = 0,\\[1mm]
F(x,v,0) =:F_0(x,v),
\end{array}\right.
\eq
where $\tilde u_F$ is given by
\bq\label{eq_tdf}
\tilde u_F(x,t) := \frac{(\psi \star b^F)(x,t) }{(\psi \star a^F)(x,t)}
\eq
with
\[
a^F(x,t) := \int_{\R^d} F(x,v,t)\,dv \quad \mbox{and} \quad b^F(x,t) := \int_{\R^d} v F(x,v,t)\,dv.
\]
The equation \eqref{kinetic-mt} equipped with the noise effect is a non-local version of our main equation \eqref{eq_1}. We now consider the singular limit in which the communication weight $\psi$ converges to a Dirac distribution, i.e., the communication rate is very concentrated around the closest neighbors of a given particle. In this framework, the Motsch-Tadmor alignment force converges to the local one \cite{KMT2}:
\[
\tilde u_F - v \to u_F - v.
\]
From this observation, it is reasonable to expect our main equation \eqref{eq_1} as the mean-field limit of the Motsch-Tadmor model with noise. It is also worth noticing that the equation \eqref{eq_1} is of the classical form, usually called the nonlinear Fokker-Planck equation \cite{Vill}, which has found applications in various fields such as plasma physics, astrophysics, the physics of polymer fluids, population dynamics, and neurophysics \cite{Frank}. We refer to \cite{CCTT, KMT, KMT3} for the existence of weak solutions and hydrodynamic limit of \eqref{eq_1}. In \cite{CCK, Choi}, the equation \eqref{eq_1} coupled to the incompressible flow are considered, and global existence of weak solutions, hydrodynamic limit, and large-time behavior are provided.

As we mentioned before, we are interested in the stability of solutions near the global Maxwellian and the rate of convergence of these solutions towards it for the Cauchy problem \eqref{eq_1}-\eqref{ini_eq_1}. For this, we define the perturbation $f = f(x,v,t)$ by
\[
F = M + \sqrt{M}f,
\]
where the global Maxwellian function
\[
M = M(v) = \frac{1}{(2\pi)^{d/2}}\exp\lt( -\frac{|v|^2}{2}\rt)
\]
is normalized to have zero bulk velocity and unit density and temperature. The equation for the perturbation $f$ satisfies 
\bq\label{main_eq}
\pa_t f + v \cdot \nabla_x f + u_F \cdot \nabla_v f = \bl f + \Gamma(f,f),
\eq
where the linear part $\bl f$ and the nonlinear part $\Gamma(f,f)$ are given by
\bq\label{oper_l}
\bl f := \frac{1}{\sqrt{M}}\nabla_v \cdot \lt(M \nabla_v \lt( \frac{f}{\sqrt{M}}\rt) \rt) \quad \mbox{and} \quad \Gamma(f,f) = u_F \cdot \lt( \frac{1}{2}vf + v\sqrt{M}\rt),
\eq
respectively. Note that $\bl f$ can be written as 
\[
\bl f = \Delta_v f + \frac14\lt( 2d - |v|^2\rt)f.
\]

Before we state our main result, we introduce several simplified notations. For functions $f(x,v)$, $g(x)$, $h(v)$, we denote by $\|f\|_{L^p}$, $\|g\|_{L^p}$, $\|h\|_{L^p_v}$ the usual $L^p(\R^d \times \R^d)$, $L^p_x(\R^d)$, $L^p_v(\R^d)$-norms, respectively. We also introduce norms $| \cdot |_\mu$ and $\| \cdot \|_{\mu}$ as follows.
\[
|f|_\mu^2 := \int_{\R^d} \lt( |\nabla_v f(v)|^2 + \mu(v)|f(v)|^2\rt)dv, \quad \mu(v) := 1 + |v|^2,
\]
and
\[
\|f\|_\mu^2 := \int_{\R^d \times \R^d} \lt( |\nabla_v f(v)|^2 + \mu(v)|f(v)|^2\rt)dvdx.
\]
$f \ls g$ represents that there exists a positive constant $C>0$ such that $f \leq C g$. We also denote by $C$ a generic positive constant depending only on the norms of the data, but independent of $T$, and drop $x$-dependence of differential operators $\nabla_x$, that is, $\pa_i f := \pa_{x_i} f $ for $1 \leq i \leq d$ and $\nabla f := \nabla_x f$. For any nonnegative integer $s \geq 0$, $H^s$ denotes the $s$-th order $L^2$ Sobolev space. $\mc^s([0,T];E)$ is the set of $s$-times continuously differentiable functions from an interval $[0,T]\subset \R$ into a Banach space $E$, and $L^p(0,T;E)$ is the set of the $L^2$ functions from an interval $(0,T)$ to a Banach space $E$. $\nabla^s$ denotes any partial derivative $\pa^\alpha$ with multi-index $\alpha, |\alpha| = s$.

\begin{theorem}\label{thm_main1} Let $d \geq 3$ and $s \geq 2[d/2] +2$. Suppose $F_0 \equiv M + \sqrt{M}f_0 \geq 0$ and $\|f_0\|_{H^s} \leq \epsilon_0 \ll 1$. Then we have the global existence of the unique classical solution $f(x,v,t)$ to the equation \eqref{main_eq} satisfying
\[
f \in \mc([0,\infty);H^s(\R^d \times \R^d)), \quad F \equiv M + \sqrt{M}f \geq 0,
\]
and
\[
\|f(t)\|_{H^s}^2 + c_1\int_0^t \|\nabla(a,b)\|_{H^{s-1}}^2 ds + c_1\int_0^t  \sum_{0 \leq k+l \leq s}\|\nabla^k \nabla_v^l \oper f\|_\mu^2 ds  \leq c_2\|f_0\|_{H^s}^2,
\]
for some positive constants $c_1,c_2 > 0$.
Furthermore, if $\|f_0\|_{L^2_v(L^1)}$ is bounded, we have
\[
\|f(t)\|_{H^s} \leq C\lt( \|f_0\|_{H^s} + \|f_0\|_{L^2_v(L^1)}\rt)(1 + t)^{-\frac d4}, \quad t \geq 0,
\]
where $C$ is a positive constant independent of $t$.
\end{theorem}
\begin{remark}\label{rmk_large} In the case of the periodic spatial domain, i.e., the particles governed by the equation \eqref{eq_1} are in the torus $\T^d$, we have the exponential decay of $\|f(t)\|_{H^s}$ under additional assumptions on the initial data. We give the details of that in Remark \ref{rmk_lt}.
\end{remark}

\begin{remark}Our strategy can also be extended to the non-local version of \eqref{eq_1}, i.e., the equation \eqref{eq_1} with $\tilde u_F$ defined in \eqref{eq_tdf}, when the communication weight $\psi$ satisfies $\psi \in L^1(\R^d)$ and $\|\psi\|_{L^1} = 1$.
\end{remark}

For the proof of Theorem \ref{thm_main1}, we employ a similar strategy as in \cite{DFT} where the global existence and large-time behavior of classical solutions to the kinetic Cucker-Smale equation with the density-dependent diffusion are investigated. As we briefly mentioned above, the alignment of Cucker-Smale equation is normalized with the total mass, and this makes the equation considered in \cite{DFT} has weak nonlinearity compared to our equation \eqref{main_eq}. To be more precise, if we similarly formulate the equation \eqref{main_eq} with the alignment of Cucker-Smale equation, the nonlinear part $\Gamma(f,f)$ of this equation becomes bilinear. However, in our situation the nonlinear operator $\Gamma(f,f)$ defined in \eqref{oper_l} is not bilinear due to the different normalization. Thus, a more delicate analysis is required for the energy estimates. Our careful analysis of the nonlinear dissipation term enables to obtain the uniform bound of a total energy functional ${\mathcal{E}}(f)$ (see \eqref{not_ed}), by which we can show that the existence of global classical solutions and its large-time behavior of the equation \eqref{main_eq}.

The rest of this paper is organized as follows. In Section \ref{sec_pre}, we recall some basic properties of the linear operator $\bl$ and hypocoercivity for a linearized Cauchy problem with a non-homogeneous microscopic source. The proof of hypocoercivity property showing the algebraic time decay rate is postponed to Appendix \ref{app_hypo} for the sake of a simpler presentation. We also provide several useful lemmas. Section \ref{sec_loc} is devoted to present the local existence and uniqueness of the equation \eqref{main_eq}. In Section \ref{sec_apri}, we show a priori estimates of solutions which consist of the classical energy estimates and the macro-micro decomposition argument. Finally, in Section \ref{sec_ext} we provide details of the proof of Theorem \ref{thm_main1}.
%
%
%
%
\section{Preliminaries}\label{sec_pre}
\subsection{Coercivity and hypocoercivity}
In this part, we recall some properties of the linear Fokker-Planck operator $\bl$ and provide the hypocoercivity for a type of linear Fokker-Planck equation with other sources. 

For this, we first decompose the Hilbert space $L^2_v$ as 
\[
L^2_v = \mathcal{N} \oplus \mathcal{N}^\perp, \quad \mathcal{N} = \mbox{span}\{ \sqrt{M}, v \sqrt{M}\}.
\]
Note that $\sqrt{M}, v_1\sqrt{M}, \cdots, v_d\sqrt{M}$ form an orthonormal basis of $\mathcal{N}$. We now define the projector $\bpp$ by
\[
\bpp : L^2_v \to \mathcal{N}, \quad f \mapsto \bpp f = \{a^f + b^f \cdot v\}\sqrt{M},
\]
then it is clear to get
\[
a^f = \lal \sqrt{M}, f\ral \quad \mbox{and} \quad b^f = \lal v\sqrt{M},f\ral,
\]
where $\lal \cdot, \cdot \ral$ denotes the inner product in the Hilbert space $L^2_v(\R^d)$. 
We also introduce projectors $\bpp_0$ and $\bpp_1$ as follows. 
\[
\bpp_0 f := a^f \sqrt{M} \quad \mbox{and} \quad \bpp_1 f:= b^f \cdot v \sqrt{M}.
\]
Then $\bpp$ can be written as 
\[
\bpp = \bpp_0 \oplus \bpp_1.
\]

In the following proposition, we provide some dissipative properties of the linear Fokker-Planck operator $\bl$. For the proof, we refer to \cite{CDM, DFT}.
\begin{proposition}The linear operator $\bl$ has the following properties. \newline
(i) 
\[
\lal \bl f, f \ral = - \int_{\R^d} \lt|\nabla_v \lt( \frac{f}{\sqrt{M}}\rt) \rt|^2 M\,dv,
\]
\[
\mbox{Ker } \bl = \mbox{Span}\{\sqrt{M} \}, \quad \mbox{and} \quad \mbox{Range } \bl = \mbox{Span}\{\sqrt{M} \}^\perp.
\]
(ii) There exists a positive constant $\lambda > 0$ such that the coercivity estimate holds:
\[
-\lal \bl f, f \ral \geq \lambda|\{ \bi - \bpp \}f|_\mu^2 + |b^f|^2.
\]
(iii) There exists a positive constant $\lambda_0 > 0$ such that 
\[
-\lal \bl f, f \ral \geq \lambda_0|\{ \bi - \bpp_0 \}f|_\mu^2.
\]
\end{proposition}
For the sake of simplicity, we drop the superscript $f$, and denote $a^f$ and $b^f$ by $a$ and $b$, respectively. 

We next show the hypocoercivity property for the Caucy problem with a non-homogeneous source. Consider the following linear Cauchy problem:
\bq\label{hyp_sys}
\pa_t f = \bb f + b \cdot v \sqrt{M}+ h, \quad x \in \R^d, \quad t > 0,
\eq
with the initial data:
\bq\label{ini_hyp_sys}
f(x,v,0) =: f_0(x,v).
\eq
Here $h = h(x,v,t)$ and $f_0 = f_0(x,v)$ are given, and the linear operator $\bb$ is defined by
\[
\bb = -v \cdot \nabla +  \bl, \quad \bl f = \Delta_v f + \frac14\lt( 2d - |v|^2\rt)f.
\]
We define $e^{\bb t}$ as the solution operator to the Cauchy problem \eqref{hyp_sys} with $h \equiv 0$, i.e., the solution to the equations \eqref{hyp_sys}-\eqref{ini_hyp_sys} can be formally written as
\[
f(t) = e^{\bb t}f_0 + \int_0^t e^{\bb(t-s)}h(s)\,ds.
\]
Before we state the estimate for algebraic decay of $e^{\bb t}$, we set the index of rate $\sigma_{d,q,m}$ by
\[
\sigma_{d,q,m} = \frac d2 \lt( \frac1q - \frac12\rt) + \frac m2.
\]
\begin{proposition}\label{prop_hypo}Let $1 \leq q \leq 2$ and $d \geq 1$. \newline

$(i)$ For any $k,l \geq 0$ with $l \leq k$, and $f_0$ satisfies $\nabla^k f_0 \in L^2(\R^d \times \R^d)$ and $\nabla^l f_0 \in L^2_v(\R^d;L^q_x(\R^d))$, we have
\[
\|\nabla^k e^{\bb t} f_0\|_{L^2} \leq C(1 + t)^{-\sigma_{d,q,m}} \lt( \|\nabla^l f_0\|_{L^2_v(L^q)} + \|\nabla^k f_0\|_{L^2} \rt), \quad t > 0,
\]
where $m = k-l$ and $C$ is a positive constant depending only on $d,m,$ and $q$. \newline

$(ii)$ For any $k,l \geq 0$ with $l \leq k$. Suppose the non-homogeneous source $h$ satisfies $\mu(v)^{-1/2}\nabla^k h(t) \in L^2(\R^d \times \R^d)$ and $\mu(v)^{-1/2}\nabla^l h(t) \in Z_q$ for $t \geq 0$. Furthermore, $h$ satisfies
\bq\label{condi_h}
\lal h, \sqrt{M} \ral = 0 \quad \mbox{and} \quad \lal h, v\sqrt{M} \ral = 0 \quad \mbox{for} \quad (x,t) \in \R^d \times \R_+,
\eq
then we have
$$\begin{aligned}
&\lt\|\nabla^k \int_0^t e^{\bb(t-s)} h(s)\,ds \rt\|_{L^2}^2\cr
&\quad \leq C\int_0^t (1 + t - s)^{-2\sigma_{d,q,m}}\lt(\|\mu^{-1/2}\nabla^l h(s)\|_{L^2_v(L^q)}^2 + \|\mu^{-1/2}\nabla^k h(s)\|_{L^2}^2 \rt)\,ds, \quad t \geq 0,
\end{aligned}$$
where $m = k-l$ and $C$ is a positive constant depending only on $d,m$, and $q$.
\end{proposition}
\begin{proof}Although the proof is similar as in \cite{Duan, DFT},  we provide the detailed proof in Appendix \ref{app_hypo} for the reader's convenience. 
\end{proof}
\subsection{Technical lemmas}
We first recall several Sobolev inequalities which will be used in the rest of this paper.
\begin{lemma}\label{lem_sob} (i) For any pair of functions $f,g \in (H^{k} \cap L^\infty)(\R^d)$, we obtain
\[
\|\nabla^k (fg)\|_{L^\infty} \ls \|f\|_{L^\infty}\|\nabla^k g\|_{L^\infty} + \|\nabla^k f\|_{L^\infty} \|g\|_{L^\infty}.
\]
Furthermore if $\nabla f \in L^\infty(\R^d)$, we have
\[
\|\nabla^k (fg) - f \nabla^k g\|_{L^2} \ls \|\nabla f\|_{L^\infty}\|\nabla^{k-1}g\|_{L^2} + \|\nabla^k f\|_{L^2}\|g\|_{L^\infty}.
\]
(ii) For $f \in H^{[d/2]+1}(\R^d)$, we have
\[
\|f\|_{L^\infty} \ls \|\nabla f\|_{H^{[d/2]}}.
\]
(iii) For $f \in (H^{k} \cap L^\infty)(\R^d)$, let $k \in \mathbb{N}, p \in [1,\infty]$, and $h\in \mc^k(B(0,\|f\|_{L^\infty}))$ where $B(0,R)$ denotes the ball of radius $R>0$ centred at the origin in $\R^d$, i.e., $B(0,R) := \{ x\in \R^d: |x| \leq R\}$. Then there exists a positive constant $c = c(k,p,h)$ such that
\[
\|\nabla^k h(f)\|_{L^p} \leq c\|f\|_{L^\infty}^{k-1}\|\nabla^k f\|_{L^p}.
\]
\end{lemma}
In the lemma below, we provide some relations between the macro components $(a,b)$ and $f$.
\begin{lemma}\label{lem_us}Let $T > 0, ~d \geq 3$, and $s \geq 2[d/2] +2$. For $0 \leq k \leq s$, we have
\[
\|\nabla^k (a,b)\|_{L^2} \leq C\|\nabla^k f\|_{L^2},
\]
where $C$ is a positive constant. Furthermore, if we assume $\sup_{0 \leq t \leq T} \|f(t)\|_{H^s} \leq \e \ll 1$, then we have
\[
\lt\| \nabla^k \lt(\frac{b}{1+a} \rt)\rt\|_{L^2} \leq C\|\nabla^k(a,b)\|_{L^2} \quad \mbox{and} \quad \lt\| \nabla^k \lt(\frac{b \otimes b}{1+a} \rt)\rt\|_{L^2} \leq C\e\|\nabla^k(a,b)\|_{L^2},
\]
for some positive constant $C > 0$.
\end{lemma}
\begin{proof}
By the definition of $a$ and $b$, it is clear to get
\bq\label{est_2_1}
\|\nabla^k (a,b)\|_{L^2} \leq C\|\nabla^k f\|_{L^2},
\eq
for some positive constant $C > 0$. Then we use \eqref{est_2_1} and the assumption on $f$ to find 
\[
\sup_{0 \leq t \leq T}\|(a(t),b(t))\|_{H^s} \leq \e \ll 1.
\]
In particular, we obtain $\|(a,b)\|_{L^\infty} \leq C\|(a,b)\|_{L^\infty(0,T;H^s)} \leq \e \ll1$, and this yields from Lemma \ref{lem_sob} that 
\[
\lt\|\nabla^k \lt( \frac{1}{1+a}\rt)\rt\|_{L^2} \ls \|\nabla^k a\|_{L^2}.
\]
This and together with the Sobolev inequality in Lemma \ref{lem_sob}, we deduce
$$\begin{aligned}
\lt\| \nabla^k \lt(\frac{b}{1+a} \rt)\rt\|_{L^2} &\ls \|b\|_{L^\infty}\lt\| \nabla^k \lt( \frac{1}{1+a}\rt)\rt\|_{L^2} + \|\nabla^k b\|_{L^2}\lt\|\frac{1}{1+a} \rt\|_{L^\infty}\cr
&\ls \|\nabla b\|_{H^{[d/2]}}\|\nabla^k a\|_{L^2} + \|\nabla^k b\|_{L^2}\cr
&\ls \|\nabla^k(a,b)\|_{L^2}.
\end{aligned}$$
Similarly, we also have
$$\begin{aligned}
\lt\| \nabla^k \lt(\frac{b \otimes b}{1+a} \rt)\rt\|_{L^2} &\ls \|b\|_{L^\infty}\lt\| \nabla^k \lt(\frac{b}{1+a} \rt)\rt\|_{L^2} + \lt\|\frac{b}{1+a}\rt\|_{L^\infty}\|\nabla^k b\|_{L^2}\cr
&\ls \|b\|_{L^\infty}\|\nabla^k (a,b)\|_{L^2}\cr
&\leq C\e\|\nabla^k (a,b)\|_{L^2}.
\end{aligned}$$
\end{proof}
Before we complete this section, we give a type of Gronwall's inequality and an elementary integral calculus whose proofs can be found in \cite{BDGM, CDM}.
\begin{lemma}\label{lem_seq}For a positive constant $T > 0$, let $\{a_n\}_{n \in \N}$ be a sequence of nonnegative continuous functions defined on $[0,T]$ satisfying 
\[
a_{n+1}(t) \leq C_1 + C_2\int_0^t a_n(s)\,ds + C_3\int_0^t a_{n+1}(s)\,ds,\quad 0 \leq t \leq T,
\]
where $C_i,i=1,2,3$ are nonnegative constants. Then there exists a positive constant $K$ such that for all $n \in \N$
\begin{displaymath}
 a_n(t) \leq \left\{ \begin{array}{ll}
 \displaystyle \frac{K^n t^n}{n!} & \textrm{if $C_1 = 0$,}\\[3mm]
Ke^{Kt} & \textrm{if $C_1 > 0$}.
  \end{array} \right.
\end{displaymath}
\end{lemma}
\begin{lemma}\label{lem_ei}For any $0 < \alpha \neq 1$ and $\beta > 1$,
\[
\int_0^t (1 + t - \tau)^{-\alpha}(1 + \tau)^{-\beta}\,d\tau \leq C(1 + t)^{-\min\{\alpha,\beta \}} \quad \mbox{for} \quad t \geq0.
\]
\end{lemma}

%
%
%
%
\section{Local existence and uniqueness}\label{sec_loc}
In this section, we investigate the local existence and uniqueness of classical solutions to the Cauchy problem \eqref{main_eq}. For this, we first present a linearized equation of \eqref{main_eq} where the averaged local velocity $u_F$ is given and satisfies certain smallness and regularity assumptions. For the linearized equation, we show the global existence and uniqueness of classical solutions. Then we construct approximate solutions $\{f^n\}_{n \in \N}$, and show that they are convergent in $L^2$-space. Finally, we conclude that the limit function is our desired solution using the Sobolev embedding theorem and the lower semicontinuity of the $H^s$-norm. 

The theorem below is our main result in this section.
\begin{theorem}\label{thm_loc} Let $d \geq 3$ and $s \geq 2[d/2] +2$. There exist positive constants $\epsilon_0, \,T_0$, and $N > 0$ such that if $\|f_0\|_{H^s} \leq \epsilon_0$ and $M + \sqrt{M}f_0 \geq 0$, then there exists the unique classical solution $f$ to the Cauchy problem \eqref{main_eq} such that $M + \sqrt{M}f \geq 0$ and $f \in \mc([0,T_0];H^s(\R^d \times \R^d))$. In particular, we have
\[
\sup_{0 \leq t \leq T_0}\|f(t)\|_{H^s} \leq N.
\]
\end{theorem}

\subsection{Local well-posedness of a linearized system} In this part, we linearize the system \eqref{main_eq} and show the invariance property of the solutions in an appropriate norm. 

For a given $\bar f \in \mc([0,T]; H^s(\R^d \times \R^d))$, we consider 
\bq\label{li-sys}
\pa_t f + v \cdot \nabla f + u_{\bar F} \cdot \nabla_v f = \bl f + \Gamma(\bar f,f),
\eq
with the initial data
\[
f(x,v,0) =: f_0(x,v),
\]
where 
\[
u_{\bar F} = \frac{\int_{\R^d} v \bar f \sqrt{M}\,dv}{1 + \int_{\R^d} \bar f \sqrt{M}\,dv} \quad \mbox{and} \quad \Gamma(\bar f, f) = u_{\bar F} \cdot \lt( \frac12 v f + v\sqrt{M}\rt).
\]
By the standard linear solvability theory, the unique solution $f \in \mc([0,T];H^s(\R^d \times \R^d))$ to \eqref{li-sys} is well-defined for any positive time $T > 0$.
For notational simplicity, we set
\[
\bar a := \int_{\R^d} \bar f \sqrt{M}\,dv \quad \mbox{and} \quad \bar b := \int_{\R^d} v \bar f \sqrt{M}\,dv.
\]
\begin{lemma}\label{lem_lin} Let $d \geq 3$ and $s \geq 2[d/2] +2$. There exist constants $\epsilon_0,\,T_0$, and $N > 0$ such that if
\[
M + \sqrt{M}f_0 \geq 0, \quad \|f_0\|_{H^s} \leq \epsilon_0, \quad \mbox{and} \quad \sup_{0 \leq t \leq T_0}\|\bar f(t)\|_{H^s} \leq N,
\]
we have $M + \sqrt{M}f \geq 0$ and $f \in \mc([0,T_0];H^s(\R^d \times \R^d))$. Furthermore, we have
\[
\sup_{0 \leq t \leq T_0}\|f(t)\|_{H^s} \leq N.
\]
\end{lemma}
\begin{proof} We first notice that the equation \eqref{li-sys} can be rewritten in terms of $F(x,v,t)$:
\[
\pa_t F + v \cdot \nabla F + u_{\bar{F}} \cdot \nabla_v F = \nabla_v \cdot \lt(\nabla_v F + vF \rt).
\]
Then by the maximum principle for the linearized equation \eqref{li-sys}, we find
\[
F \equiv M + \sqrt{M}f \geq 0.
\]
We next show the bound on $f$ in the rest of the proof. It follows from \eqref{li-sys} that for $0 \leq k \leq s$
\begin{align*}
\begin{aligned}
&\frac12\frac{d}{dt} \|\nabla^k f\|_{L^2}^2 + \lambda_0 \|\{ \bi - \bpp_0\} \nabla^k f\|_{\mu}^2\cr
& \quad \leq \int_{\R^d} |\lal \nabla^k \Gamma(\bar f ,f),\nabla^k f\ral |\,dx + \sum_{0\leq l <k}\binom{k}{l}\int_{\R^d} |\lal \nabla^{k-l} u_{\bar F} \cdot \nabla_v \nabla^l f, \nabla^k f\ral |\,dx\cr
&\quad \leq \frac12 \sum_{0 \leq l \leq k}\binom{k}{l}\int_{\R^d} |\lal \nabla^{k-l} u_{\bar F} \cdot v \nabla^l f, \nabla^k f\ral |\,dx + \int_{\R^d} |\lal \nabla^k u_{\bar F}\cdot v\sqrt{M}, \nabla^k f\ral |\,dx\cr
&\qquad + \sum_{0\leq l <k}\binom{k}{l}\int_{\R^d} |\lal \nabla^{k-l} u_{\bar F} \cdot \nabla_v \nabla^l f, \nabla^k f\ral |\,dx\cr
&\quad =:\sum_{i=1}^3 I_i.
\end{aligned}
\end{align*}
First, we can easily find
\[
I_2  = \int_{\R^d} |\nabla^k u_{\bar F} \cdot \nabla^k b|\,dx \leq \|\nabla^k u_{\bar F}\|_{L^2}\|\nabla^k b\|_{L^2} \ls \|\nabla^k(\bar a, \bar b)\|_{L^2}\|\nabla^k b\|_{L^2} \ls \|\nabla^k \bar f\|_{L^2}\|\nabla^k f\|_{L^2},
\]
due to Lemma \ref{lem_us}. For the estimate of $I_1$, we obtain
\[
I_1 \ls \sum_{0\leq l \leq k}\int_{\R^d \times \R^d} |\nabla^{k-l}u_{\bar F}| |v| |\nabla^l f| |\nabla^k f|\,dxdv \leq \lt(\int_{\R^d} |\nabla^{k-l}u_{\bar F}|^2\|\nabla^l f\|_{L^2_v}^2\,dx \rt)^{1/2}\|\nabla^k f\|_{\mu}.
\]
On the other hand, the last term of the above inequality is bounded by
\bq\label{u_est1}
\lt(\int_{\R^d} |\nabla^{k-l}u_{\bar F}|^2\|\nabla^l f\|_{L^2_v}^2\,dx \rt)^{1/2} \leq \left\{ \begin{array}{ll}
\|\nabla^{k-l}u_{\bar F}\|_{L^\infty}\|\nabla^l f\|_{L^2} & \textrm{for $|k - l | \leq [d/2]$,}\\[2mm]
\|\nabla^{k-l} u_{\bar F}\|_{L^2}\|\nabla^l f\|_{L^2_v(L^\infty)} & \textrm{for $|k - l | \geq [d/2] + 1$}.
  \end{array} \right.
\eq
This implies
\[
I_1 \ls \|\bar f\|_{L^2_v(H^s)}\sum_{0 \leq \alpha + \beta \leq s}\|\nabla^\alpha \nabla^\beta_v f\|_{\mu}^2.
\]
Similarly, we also deduce 
\[
I_3 \ls \|\bar f\|_{L^2_v(H^s)}\sum_{0 \leq \alpha + \beta \leq s}\|\nabla^\alpha \nabla^\beta_v f\|_{\mu}^2.
\]
Thus we have
\bq\label{est_3_1}
\frac12\frac{d}{dt} \|\nabla^k f\|_{L^2}^2 + \lambda_0 \|\{ \bi - \bpp_0\}\nabla^k f\|_{\mu}^2 \ls \|\nabla^k \bar f\|_{L^2}\|\nabla^k f\|_{L^2} + \|\bar f\|_{L^2_v(H^s)}\sum_{0 \leq \alpha + \beta \leq s}\|\nabla^\alpha \nabla^\beta_v f\|_{\mu}^2.
\eq
Note that $\|\bpp_0 \nabla^k f\|_\mu$ can be estimated as 
$$\begin{aligned}
\|\bpp_0\nabla^k  f\|_{\mu}^2 &= \|\nabla^k a \sqrt{M}\|_{\mu}^2\cr
&= \int_{\R^d\times \R^d} |\nabla_v \nabla^k a \sqrt{M}|^2 + (1 + |v|^2)|\nabla^k a \sqrt{M}|^2\,dvdx\cr
&\leq 2 \int_{\R^d} |\nabla^k a|^2\,dx\int_{\R^d} (1+|v|^2)M\,dv\cr
&\leq C\|\nabla^k a\|_{L^2}^2\cr
&\leq C\|\nabla^k f\|_{L^2}^2,
\end{aligned}$$
and this yields
\bq\label{est_3_2}
\|\nabla^k f\|_{\mu}^2 \leq \|\{ \bi - \bpp_0\} \nabla^k f\|_{\mu}^2 + \| \bpp_0 \nabla^k f\|_{\mu}^2 \leq \|\{ \bi - \bpp_0\} \nabla^k f\|_{\mu}^2 + C\|\nabla^k f\|_{L^2}^2.
\eq
Hence, by combining \eqref{est_3_1} and \eqref{est_3_2} we have
$$\begin{aligned}
&\frac12\frac{d}{dt} \|\nabla^k f\|_{L^2}^2 + \lambda_0 \|\nabla^k f\|_{\mu}^2 \cr
&\quad \ls \|\nabla^k \bar f\|_{L^2}\|\nabla^k f\|_{L^2} + \|\bar f\|_{L^2_v(H^s)}\sum_{0 \leq \alpha + \beta \leq s}\|\nabla^\alpha \nabla^\beta_v f\|_{\mu}^2 + C\|\nabla^k f\|_{L^2}^2.
\end{aligned}$$
In a similar way, we also find 
$$\begin{aligned}
&\frac{d}{dt}\|f\|_{H^s}^2 + \lambda_0 \sum_{0\leq \alpha + \beta \leq s}\|\nabla^\alpha \nabla^\beta_v f\|_{\mu}^2\cr
&\qquad \leq C\|\bar f\|_{L^2_v(H^s)} \sum_{0 \leq \alpha + \beta \leq s}\|\nabla^\alpha \nabla^\beta_v f\|_{\mu}^2 + C\|\bar f\|_{H^s}^2 + C\|f\|_{H^s}^2.
\end{aligned}$$
Since $\|\bar f(t)\|_{H^s} \leq N$ for $0 \leq t \leq T_0$,
\[
\frac{d}{dt}\|f\|_{H^s}^2 + \lt( \lambda_0 - CN\rt)\sum_{0\leq \alpha + \beta \leq s}\|\nabla^\alpha \nabla^\beta_v f\|_{\mu}^2 \leq CN^2 + C\|f\|_{H^s}^2,
\]
and by integrating it over $[0,t]\,(t \leq T_0)$ we get
\[
\|f(t)\|_{H^s}^2 + \lt( \lambda_0 - CN\rt)\int_0^t \sum_{0\leq \alpha + \beta \leq s}\|\nabla^\alpha \nabla^\beta_v f(\tau)\|_{\mu}^2 \,d\tau \leq \epsilon_0^2 + CN^2 T_0 + C\sup_{0\leq \tau \leq T_0}\|f(\tau)\|_{H^s}^2 T_0.
\]
Finally, we choose $\epsilon_0,\, T_0 >0$, and $N > 0$ such that
\[
0 < \epsilon_0 \ll 1, \quad \frac{\lambda_0}{C} > N, \quad \mbox{and} \quad CT_0 \leq \frac14,
\]
to conclude
\[
\sup_{0\leq t \leq T_0}\|f(t)\|_{H^s}^2 \leq \epsilon_0^2 + CN^2 T_0 \leq N^2.
\]
\end{proof}
\subsection{Proof of Theorem \ref{thm_loc}} We construct the approximated solutions $f^m$ for the system \eqref{main_eq} as follows.
\bq\label{app-sys}
\pa_t f^{m+1} + v \cdot \nabla_x f^{m+1} + u_{F^m} \cdot \nabla f^{m+1} = \bl f^{m+1} + \Gamma(f^m,f^{m+1}),
\eq
with initial data and first iteration step:
\[
f^m(x,v)|_{t=0} = f_0 \quad \mbox{for all} \quad m \geq 1,\quad  (x,v) \in \R^d \times \R^d,
\]
and
\[
f^0(x,v,t) = 0 \quad \mbox{for} \quad (x,v,t) \in \R^d \times \R^d \times \R_+.
\]
Here $u_{F^m}$ and $\Gamma(f^m, f^{m+1})$ are given by
\[
u_{F^m} = \frac{\int_{\R^d} v f^m \sqrt{M}\,dv}{1 + \int_{\R^d} f^m \sqrt{M}\,dv} \quad \mbox{and} \quad \Gamma(f^m, f^{m+1}) = u_{F^m} \cdot \lt( \frac12 v f^{m+1} + v\sqrt{M}\rt),
\]
respectively. Set
\[
a^m := \int_{\R^d} f^m \sqrt{M}\,dv \quad \mbox{and} \quad b^m := \int_{\R^d} v f^m \sqrt{M}\,dv.
\]
Let $\mathcal{I}(s,T;N)$ be the solution space for $f$ defined by
\[
\mathcal{I}(s,T;N) := \lt\{ f \in \mc([0,T]; H^s(\R^d \times \R^d)) : M + \sqrt{M}f \geq 0 \quad \mbox{and} \quad \sup_{0 \leq t \leq T}\|f(t)\|_{H^s} \leq N \rt\}.
\]
Then as a direct consequence of Lemma \ref{lem_lin}, we have the uniform bound estimate of $f$ in the proposition below.
\begin{proposition}Let $d \geq 3$ and $s \geq 2[d/2] +2$. There exist positive constants $\epsilon_0,\, T_0$, and $N > 0$ such that if $f_0 \in H^s(\R^d \times \R^d)$ with $\|f_0\|_{H^s} \leq \epsilon_0$ and $M + \sqrt{M}f_0 \geq 0$, then for each $m \geq 0$, $f^m$ is well-defined and $f^m \in \mathcal{I}(s,T_0;N)$.
\end{proposition}
We next show that the approximations $f^m$ are Cauchy sequences in $\mc([0,T_0];L^2(\R^d \times \R^d))$. 
\begin{lemma}\label{lem_cau}Let $f^m$ be a sequence of approximated solutions with the initial data $f_0$ satisfying $\|f_0\|_{H^s} \leq \epsilon_0$. Then $f^m$ is Cauchy sequence in $\mc([0,T_0];L^2(\R^d \times \R^d))$.
\end{lemma}
\begin{proof} It follows from \eqref{app-sys} that
$$\begin{aligned}
&\pa_t (f^{m+1} - f^m) + v \cdot \nabla (f^{m+1} - f^m) + u_{F^m}\cdot \nabla_v (f^{m+1} - f^m)\cr
&\quad = - (u_{F^m} - u_{F^{m-1}})\cdot \nabla_v f^m + \bl (f^{m+1} - f^m) + \frac12 u_{F^{m-1}} \cdot v \lt(f^{m+1} - f^m \rt)\cr
&\qquad + (u_{F^m} - u_{F^{m-1}}) \cdot \lt(\frac12 v f^m + v \sqrt{M} \rt).
\end{aligned}$$
Then by using similar estimates as in Lemma \ref{lem_lin} with $k=0$ we get
$$\begin{aligned}
&\frac12\frac{d}{dt}\|f^{m+1} - f^m\|_{L^2}^2 + \lambda_0 \|\{ \bi - \bpp_0\} (f^{m+1} - f^m)\|_{\mu}^2\cr
&\quad \leq \frac12 \int_{\R^d} u_{F^{m-1}}\cdot \lal v (f^{m+1} - f^m), f^{m+1} - f^m\ral\,dx\cr
&\qquad + \int_{\R^d} (u_{F^m} - u_{F^{m-1}}) \cdot \lt\lal \frac12 v f^m + v \sqrt{M} - \nabla_v f^m, f^{m+1} - f^m \rt\ral\,dx\cr
&\quad \leq C\|f^{m+1} - f^m\|_{L^2}\|f^{m+1} - f^m\|_{\mu} + C\|f^m - f^{m-1}\|_{L^2}\|f^{m+1} - f^m\|_{\mu},
\end{aligned}$$
where we used
\[
\|u_{F^m} - u_{F^{m-1}}\|_{L^2} \ls \|a^m - a^{m-1}\|_{L^2} + \|b^m - b^{m-1}\|_{L^2} \ls \|f^m - f^{m-1}\|_{L^2}.
\]
We now use the estimate
\[
\|f^{m+1} - f^m\|_{\mu} \leq \|\{ \bi - \bpp_0\} (f^{m+1} - f^m)\|_{\mu}^2 + C\|f^{m+1} - f^m\|_{L^2}
\]
to obtain
\begin{align}\label{est_3_3}
\begin{aligned}
&\frac12\frac{d}{dt}\|f^{m+1} - f^m\|_{L^2}^2 + \lambda_0\|f^{m+1} - f^m\|_{\mu}^2 \cr
&\quad \leq C\|f^{m+1} - f^m\|_{L^2}^2 + C\|f^m - f^{m-1}\|_{L^2}^2 + \frac{\lambda_0}{2}\|f^{m+1} - f^m\|_{\mu}^2.
\end{aligned}
\end{align}
Applying the Gronwall's inequality for \eqref{est_3_3}, we deduce
\[
\|(f^{m+1} - f^m)(t)\|_{L^2}^2 \ls \int_0^t \|(f^m - f^{m-1})(s)\|_{L^2}^2\,ds.
\]
Finally, we use Lemma \ref{lem_seq} to have
\[
\|(f^{m+1} - f^m)(t)\|_{L^2}^2 \ls \frac{T_0^{m+1}}{(m+1)!} \quad \mbox{for} \quad t \in [0,T_0].
\]
This concludes the desired result.
\end{proof}
\begin{proof}[Proof of Theorem \ref{thm_loc}]Note that $f^m$ is bounded in $\mc([0,T_0]; H^s(\R^d \times \R^d))$. Thus we use the Gagliardo-Nirenberg interpolation inequality together with the convergence estimate in Lemma \ref{lem_cau} to deduce
\[
f^m \to f \quad \mbox{in } \mc([0,T_0];H^{s-1}(\R^d \times \R^d)),
\]
as $m \to \infty$. Furthermore, it follows from the lower semicontinuity of the norm that $f^m \in \mathcal{I}(s,T_0;N)$ implies
\[
M + \sqrt{M}f \geq 0 \quad \mbox{and} \quad \sup_{0 \leq t \leq T_0}\|f(t)\|_{H^s} \leq N. 
\]
Finally, if we let $f,g$ be the classical solutions obtained from the above with the same initial data. Then we have
\[
\|f(t) - g(t)\|_{L^2}^2 \leq C\int_0^t \|f(\tau) - g(\tau)\|_{L^2}^2\,d\tau,
\]
for some positive constant $C >0$, and the standard argument yields the uniqueness of the classical solutions. 
\end{proof}
%
%
%
%

\section{A priori estimates}\label{sec_apri}
In this section, we provide a priori estimates for the global existence of classical solutions to the equation \eqref{main_eq}. By using the classical energy method together with the careful analysis of the local averaged velocity, we obtain several uniform a priori estimates of energy inequalities. These energy estimates play a crucial role in obtaining the global well-posedness of solutions with the help of the local existence as well as the continuum argument.
\subsection{Energy estimates}
We first present a priori estimate of $\|f\|_{L^2_v(H^s)}$ for the equation \eqref{main_eq}.
\begin{lemma}\label{lem_energy1}Let $d \geq 3$ and $s \geq 2[d/2] +2$, and let $T > 0$ be given. Suppose that $f\in \mc([0,T]; H^s(\R^d \times \R^d))$ is the solution to the equation \eqref{main_eq} satisfying
\[
\sup_{0 \leq t \leq T}\|f(t)\|_{H^s} \leq \epsilon_1 \ll 1.
\]
Then we have
\[
\frac{d}{dt} \|f\|_{L^2_v(H^s)}^2 + C_1\sum_{0 \leq k \leq s}\|\nabla^k \oper f\|_{\mu}^2\leq C\epsilon_1\|\nabla(a,b)\|_{H^{s-1}}^2,
\]
for some positive constants $C_1, C > 0$.
\end{lemma}
\begin{proof} $\diamond$ $L^2$-estimate: We first easily find
$$\begin{aligned}
\frac12 \frac{d}{dt}\int_{\R^d \times \R^d} f^2\,dxdv &= \int_{\R^d \times \R^d} f\lt( - v \cdot \nabla_x f + u_F \cdot \lt( v\sqrt{M} + \frac{v}{2}f - \nabla_v f \rt) + \bl f \rt)dxdv\cr
&=\int_{\R^d \times \R^d} u_F \cdot v \sqrt{M}f\,dxdv + \int_{\R^d \times \R^d} u_F \cdot \frac{v}{2}f^2 \,dxdv + \int_{\R^d \times \R^d} f\bl f\,dxdv\cr
&=: \sum_{i=1}^3 I_i.
\end{aligned}$$
Here, $I_1$ and $I_3$ are estimated by
$$\begin{aligned}
I_1 &= \int_{\R^d} u_F \cdot b\,dx = \int_{\R^d} \frac{|b|^2}{1 + a}\,dx,\cr
I_3 &= \int_{\R^d \times \R^d} f \,\bl \oper f \,dxdv + \int_{\R^d \times \R^d} f \,\bl \bpp f\,dxdv = \int_{\R^d} \lal \bl \oper f, f\ral dx - \int_{\R^d} |b|^2 dx.
\end{aligned}$$
Thus, we obtain
\[
I_1 + I_3 =  \int_{\R^d} \lal \bl \oper f, f\ral dx - \int_{\R^d} \frac{a|b|^2}{1 + a}dx.
\]
For the estimate of $I_2$, we notice that
$$\begin{aligned}
\lal vf,f\ral &= \lal v, |\bpp f|^2 \ral + 2\lal v\bpp f, \oper f\ral + \lal v, |\oper f|^2 \ral\cr
&= 2ab + 2\lal v\bpp f, \oper f\ral + \lal v, |\oper f|^2 \ral.
\end{aligned}$$
This deduces
$$\begin{aligned}
\frac12 \frac{d}{dt}\int_{\R^d \times \R^d} f^2\,dxdv &= \int_{\R^d} \lal \bl \oper f, f\ral dx + \int_{\R^d} u_F \cdot \lal v \bpp f, \oper f\ral dx\cr
&\quad + \frac12 \int_{\R^d} u_F \cdot \lal v | \oper f|^2 \ral dx\cr
&\leq \int_{\R^d} \lal \bl \oper f, f\ral dx + \|u_F\|_{L^\infty}\|\bpp f\|_{L^2}\|\oper f\|_{\mu}\cr
&\quad + \|u_F\|_{L^\infty}\|\oper f\|_{\mu}^2\cr
&\leq \int_{\R^d} \lal \bl \oper f, f\ral dx + C\epsilon_1\|\nabla(a,b)\|_{H^{s-1}}^2 + C\epsilon_1\|\oper f\|_{\mu}^2,
\end{aligned}$$
where we used
\[
\|u_F\|_{L^\infty} \leq C\|\nabla u_F\|_{H^{[d/2]}} \leq C\|\nabla(a,b)\|_{H^{s-1}} \quad \mbox{and} \quad \|\bpp f\|_{L^2} \leq C\|(a,b)\|_{L^2} \leq C\epsilon_1,
\]
due to Lemma \ref{lem_us}.
Thus we have
\[
\frac{d}{dt}\|f\|_{L^2}^2 + \lambda\|\oper f\|_{\mu}^2 \leq C\epsilon_1\|\nabla(a,b)\|_{H^{s-1}}^2.
\]
$\diamond$ $H^s_x$-estimate: For $1 \leq k \leq s$, we take $\nabla^k$ to \eqref{main_eq} to get
\[
\pa_t \nabla^k f + v \cdot \nabla^{k+1} f + \nabla^k (u_F \cdot \nabla_v f) = \bl \nabla^k f + \nabla^k \Gamma(f,f),
\]
where 
\[
\Gamma(f,f) = u_F \cdot v\sqrt{M} + \frac{u_F \cdot v}{2}f.
\]
Then we find
$$\begin{aligned}
\frac12\frac{d}{dt}\|\nabla^k f\|_{L^2}^2 &= \int_{\R^d} \lal \bl \nabla^k f, \nabla^k f\ral dx + \int_{\R^d} \lal \nabla^k \Gamma(f,f), \nabla^k f\ral dx \cr
&\quad - \sum_{0 \leq l < k}\binom{k}{l}\int_{\R^d} \lal \nabla^{k-l}u_F \cdot \nabla_v \nabla^l f,\nabla^k f\ral dx\cr
&=: \sum_{i=1}^3J_i,
\end{aligned}$$
where $J_1$ is easily estimated by
\begin{align}\label{est_e1}
\begin{aligned}
J_1 &= \int_{\R^d} \lal \bl \nabla^k \oper f, \nabla^k f\ral dx + \int_{\R^d} \lal \bl \nabla^k \bpp f, \nabla^k f\ral dx\cr
&= \int_{\R^d} \lal \bl \nabla^k \oper f, \nabla^k f\ral dx - \int_{\R^d} |\nabla^k b|^2 dx.
\end{aligned}
\end{align}
For the estimate of $J_2$, we decompose it into two terms:
$$\begin{aligned}
J_2 &= \int_{\R^d} \lal \nabla^k (u_F \cdot v \sqrt{M}),\nabla^k f\ral dx + \frac12 \int_{\R^d} \lal \nabla^k (u_F \cdot vf),\nabla^k f\ral dx\cr
&=:J_2^1 + J_2^2.
\end{aligned}$$
Here, $J_2^1$ is estimated as 
$$\begin{aligned}
J_2^1 &= \int_{\R^d} \nabla^k u_F \cdot \nabla^k b\,dx\cr
&= \int_{\R^d} \frac{|\nabla^k b|^2}{1 + a}dx + \sum_{0 \leq l < k} \binom{k}{l}\int_{\R^d} \nabla^l b \nabla^{k-l}\lt( \frac{1}{1+a}\rt) \cdot \nabla^k b\,dx\cr
&\leq \int_{\R^d} \frac{|\nabla^k b|^2}{1 + a}dx + C\epsilon_1\|\nabla (a,b)\|_{H^{s-1}}^2,
\end{aligned}$$
where we used for $|k-l| \geq [d/2] + 1$
$$\begin{aligned}
\int_{\R^d} \nabla^l b \nabla^{k-l}\lt( \frac{1}{1+a}\rt) \cdot \nabla^k b\,dx &\leq \|\nabla^l b\|_{L^\infty}\lt\| \nabla^{k-l} \lt( \frac{1}{1+a}\rt)\rt\|_{L^2}\|\nabla^k b\|_{L^2}\cr
&\ls \|\nabla^{l + 1} b\|_{H^{[d/2]}}\|\nabla^{k-l} a\|_{L^2}\|\nabla^k b\|_{L^2}\cr
&\leq C\epsilon_1\|\nabla(a,b)\|_{H^{s-1}}^2,
\end{aligned}$$
and for $|k-l| \leq [d/2]$
$$\begin{aligned}
\int_{\R^d} \nabla^l b \nabla^{k-l}\lt( \frac{1}{1+a}\rt) \cdot \nabla^k b\,dx &\leq \|\nabla^l b\|_{L^2}\lt\| \nabla^{k-l} \lt( \frac{1}{1+a}\rt)\rt\|_{L^\infty}\|\nabla^k b\|_{L^2}\cr
&\ls \|\nabla^{l} b\|_{L^2}\|\nabla^{k-l+1} a\|_{H^{[d/2]}}\|\nabla^k b\|_{L^2}\cr
&\leq C\epsilon_1\|\nabla(a,b)\|_{H^{s-1}}^2.
\end{aligned}$$
Similarly, we obtain
$$\begin{aligned}
J_2^2 &= \frac12 \sum_{0 \leq l \leq k}\binom{k}{l}\int_{\R^d} \lal \nabla^{k-l}u_F \cdot v\lt( \{ \bi - \bpp\}\nabla^l f+ \bpp \nabla^l f \rt), \{ \bi - \bpp\}\nabla^k f+ \bpp \nabla^k f\ral \,dx\cr
&\leq C\epsilon_1 \lt( \sum_{1 \leq k \leq s} \|\nabla^k \oper f\|_{\mu}^2 + \|\nabla (a,b)\|_{H^{s-1}}^2\rt).
\end{aligned}$$
This yields
\bq\label{est_e2}
J_2 \leq \int_{\R^d} \frac{|\nabla^k b|^2}{1 + a}dx + C\epsilon_1 \lt( \sum_{1 \leq k \leq s} \|\nabla^k \oper f\|_{\mu}^2 + \|\nabla (a,b)\|_{H^{s-1}}^2\rt).
\eq
Finally, we again use similar arguments as the above to find
\bq\label{est_e3}
J_3 \leq C\epsilon_1 \lt( \sum_{1 \leq k \leq s} \|\nabla^k \oper f\|_{\mu}^2 + \|\nabla (a,b)\|_{H^{s-1}}^2\rt).
\eq
Hence by combining \eqref{est_e1}-\eqref{est_e3} we have
$$\begin{aligned}
\frac{d}{dt} \|\nabla f\|_{L^2_v(H^{s-1})}^2 + \lt(\lambda - C\epsilon_1\rt)\sum_{1 \leq k \leq s}\|\nabla^k \oper f\|_{\mu}^2 &\leq -\sum_{1 \leq k \leq s}\int_{\R^d} \frac{a}{1+a}|\nabla^k b|^2dx + C\e\|\nabla(a,b)\|_{H^{s-1}}^2\cr
&\leq C\epsilon_1\|\nabla(a,b)\|_{H^{s-1}}^2.
\end{aligned}$$
\end{proof}
We next provide the mixed space-velocity derivative of $\oper f$ in the following lemma.
\begin{lemma}\label{lem_energy2}
Let $d \geq 3$ and $s \geq 2[d/2] +2$, and let $T > 0$ be given. Suppose that $f\in \mc([0,T]; H^s(\R^d \times \R^d))$ is the solution to the equation \eqref{main_eq} satisfying
\[
\sup_{0 \leq t \leq T}\|f(t)\|_{H^s} \leq \epsilon_1 \ll 1.
\]
Then for fixed $1 \leq l \leq s$ we have
\begin{align}\label{energy2}
\begin{aligned}
&\frac{d}{dt} \sum_{0 \leq k \leq s - l} \|\nabla^k \nabla^l_v \oper f\|_{L^2}^2 + C_2 \sum_{0 \leq k \leq s - l}\|\nabla^k \nabla^l_v \oper f\|_\mu^2\cr
&\qquad \leq C\epsilon_1 \lt( \sum_{0 \leq k \leq s-l}\|\nabla^k \nabla^l_v \oper f\|_\mu^2 + \|\nabla(a,b)\|_{H^{s-1}}^2\rt)\cr
&\qquad \quad + C\lt( \sum_{0 \leq k \leq s+1 - l}\|\nabla^k \oper f\|_\mu^2 + \|\nabla(a,b)\|_{H^{s-l}}^2\rt) \cr
&\qquad \quad+ C\chi_{\{2 \leq l \leq s\}}\sum_{\substack{1 \leq \beta \leq l-1 \\ 0 \leq \alpha + \beta \leq s}}\|\nabla^\alpha \nabla^\beta_v \oper f\|_\mu^2,
\end{aligned}
\end{align}
for some positive constants $C_2, C > 0$. Here $\chi_A$ denotes the characteristic function of a set $A$.
\end{lemma}
\begin{proof}By applying $\oper$ to the equation \eqref{main_eq}, we find
\[
\pa_t \oper f + \oper \lt(v \cdot \nabla f + u_F \cdot \nabla_v f \rt) = \oper \bl f + \oper \Gamma (f,f).
\]
Since $\{\bi - \bpp\}(v\sqrt{M}) = 0$, we get
$$\begin{aligned}
\oper \Gamma(f,f) &= \frac12 u_F \cdot v \oper f + \frac12 u_F \cdot [\oper,v]f\cr
&= \frac12 u_F \cdot v \oper f + \frac12 u_F \cdot [v,\bpp]f.
\end{aligned}$$
This deduces
$$\begin{aligned}
&\pa_t \oper f + v \cdot \nabla \oper f + u_F \cdot \nabla_v \oper f\cr
&\qquad = \bl \oper f + \frac12 u_F \cdot v \oper f + \frac12 u_F \cdot [v,\bpp]f\cr
&\qquad \quad + \bpp \lt( v \cdot \nabla \oper f + u_F \cdot \nabla_v \oper f\rt) - \oper \lt(v \cdot \nabla \bpp f + u_F \cdot \nabla_v \bpp f \rt),
\end{aligned}$$
where we used $\oper \bl = \bl \oper$.
Then we have for $0 \leq k + l \leq s$ with fixed $1 \leq l \leq s$
$$\begin{aligned}
&\frac12\frac{d}{dt}\|\nabla^k \nabla_v^l \oper f\|_{L^2}^2 \cr
& = -\int_{\R^d} \lal \nabla^k\nabla_v^l(v \cdot \nabla \oper f) + \nabla^k\nabla_v^l(u_F \cdot \nabla \oper f), \nabla^k \nabla_v^l \oper f \ral\,dx\cr
&\quad +\int_{\R^d} \lal \nabla^k\nabla_v^l(\bl\oper f) + \frac12  \nabla^k\nabla_v^l(u_F \cdot v\oper f) + \frac12\nabla^k\nabla_v^l(u_F \cdot [v,\bpp] f), \nabla^k \nabla_v^l \oper f \ral\,dx\cr
&\quad + \int_{\R^d} \lal \nabla^k\nabla_v^l(\bpp\lt( v \cdot \nabla\oper f + u_F \cdot \nabla_v \oper f\rt)), \nabla^k \nabla_v^l \oper f \ral\,dx\cr
&\quad + \int_{\R^d} \lal \nabla^k\nabla_v^l(\oper \lt( v \cdot \nabla \bpp f + u_F \cdot \nabla_v \bpp f\rt)), \nabla^k \nabla_v^l \oper f \ral\,dx\cr
&=:\sum_{i=1}^9 K_i.
\end{aligned}$$
$\diamond$ Estimate of $K_1$: A straightforward computation yields
$$\begin{aligned}
K_1 &= -\int_{\R^d} \lal \nabla^k [\nabla^l_v,v \cdot \nabla]\oper f, \nabla^k \nabla_v^l \oper f \ral\,dx\cr
&\leq C\|[\nabla^l_v,v \cdot \nabla]\nabla^k \oper f\|_{L^2}^2 + \delta \|\nabla^k \nabla_v^l \oper f\|_{L^2}^2\cr
&\leq C\sum_{0 \leq \alpha \leq s - l}\|\nabla^{\alpha+1}\oper f\|_{L^2}^2 + C\chi_{\{2 \leq l \leq s\}}\sum_{\substack{ 1 \leq \beta \leq l-1 \\ 0 \leq \alpha + \beta \leq s}}\|\nabla^\alpha \nabla^\beta_v\oper f\|_{L^2}^2 \cr
&\quad + \delta \|\nabla^k \nabla_v^l \oper f\|_{L^2}^2,
\end{aligned}$$
where $\delta>0$ is a positive constant which will be determined later, and $[\cdot,\cdot]$ denotes the commutator operator, i.e., $[A,B]:=AB-BA$. \newline

\noindent $\diamond$ Estimate of $K_2$: Similar to \eqref{u_est1}, we find
$$\begin{aligned}
K_2 &= \sum_{0 \leq \alpha < k}\binom{k}{\alpha}\int_{\R^d} \lal \nabla^{k-\alpha} u_F \cdot \nabla_v \nabla^\alpha\nabla^l_v\oper f, \nabla^k \nabla_v^l \oper f \ral\,dx\cr
&\leq C\sum_{0 \leq \alpha < k} \int_{\R^d} |\nabla^{k-\alpha} u_F|\| \nabla^\alpha\nabla^{l+1}_v\oper f \|_{L^2_v}\|\nabla^k \nabla_v^l \oper f \|_{L^2_v}\,dx\cr
&\leq C\|\nabla (a,b)\|_{H^{s-1}}\sum_{0 \leq k + l \leq s}\|\nabla^k \nabla^l_v \oper f\|_{L^2}^2\cr
&\leq C\epsilon_1\sum_{0 \leq k + l \leq s}\|\nabla^k \nabla^l_v \oper f\|_{L^2}^2.
\end{aligned}$$
$\diamond$ Estimate of $K_3$: Since 
\[
[\nabla^l_v, \bl] = [\nabla^l_v, \Delta_v] + [\nabla^l_v, \frac14(2d - |v|^2)] = [\nabla^l_v, -|v|^2],
\]
we obtain
$$\begin{aligned}
K_3 &= \int_{\R^d} \lal \nabla^k [\nabla^l_v, \bl]\oper f, \nabla^k \nabla_v^l \oper f\ral\,dx - \int_{\R^d} \lal \bl \nabla^k \nabla_v^l \oper f, \nabla^k \nabla_v^l \oper f\ral\,dx\cr
&= \int_{\R^d} \lal \nabla^k [\nabla^l_v, -|v|^2]\oper f, \nabla^k \nabla_v^l \oper f\ral\,dx - \int_{\R^d} \lal \bl \nabla^k \nabla_v^l \oper f, \nabla^k \nabla_v^l \oper f\ral\,dx.
\end{aligned}$$
On the other hand, the first term of the above equality can be estimated as
$$\begin{aligned}
&\int_{\R^d} \lal \nabla^k [\nabla^l_v, -|v|^2]\oper f, \nabla^k \nabla_v^l \oper f\ral\,dx\cr
&\quad \leq  C\|[\nabla^l_v, -|v|^2]\nabla^k \oper f\|_{L^2}^2 + \delta \|\nabla^k \nabla_v^l \oper f\|_{L^2}^2\cr
&\quad \leq C\sum_{0 \leq \alpha \leq s - l}\|\nabla^{\alpha+1}\oper f\|_{\mu}^2 + C\chi_{\{2 \leq l \leq s\}}\sum_{\substack{ 1 \leq \beta \leq l-1 \\ 0 \leq \alpha + \beta \leq s}}\|\nabla^\alpha \nabla^\beta_v\oper f\|_{\mu}^2 \cr
&\qquad + \delta \|\nabla^k \nabla_v^l \oper f\|_{L^2}^2.
\end{aligned}$$
$\diamond$ Estimates of $K_4$ and $K_5$: Similar to the estimate of $K_2$, we get
$$\begin{aligned}
K_4 & \ls\sum_{0 \leq \alpha \leq k}\int_{\R^d} |\nabla^{k - \alpha} u_F| \|\nabla^\alpha\nabla^l_v(v \oper f)\|_{L^2_v} \|\nabla^k\nabla^l_v \oper f\|_{L^2_v}\,dx\cr
&\ls\|(a,b)\|_{H^s} \sum_{0 \leq \alpha + \beta \leq s} \|\nabla^\alpha \nabla^\beta_v\oper f\|_{\mu}^2  \cr
&\leq C\epsilon_1\sum_{0 \leq \alpha + \beta \leq s} \|\nabla^\alpha \nabla^\beta_v\oper f\|_{\mu}^2,\cr
K_5 &\ls \sum_{0 \leq \alpha \leq k} \int_{\R^d} | \nabla^{k-\alpha} u_F|\|\nabla^l_v([v,\bpp]\nabla^\alpha f)\|_{L^2_v}\|\nabla^k \nabla^l_v \oper f\|_{L^2_v}\,dx\cr
&\ls \|(a,b)\|_{H^s}\lt(\|\nabla(a,b)\|_{H^{s-1}}^2 + \|\nabla^k \nabla^l_v\oper f\|_{L^2}^2\rt)\cr
&\leq C\epsilon_1\lt(\|\nabla(a,b)\|_{H^{s-1}}^2 + \|\nabla^k \nabla^l_v\oper f\|_{L^2}^2\rt).
\end{aligned}$$
$\diamond$ Estimates of $K_i,i=6,\cdots,9$: Similarly, we have
$$\begin{aligned}
K_6 + K_7 &\leq C\sum_{0 \leq \alpha \leq s - l}\|\nabla^{\alpha + 1}\oper f\|_{L^2}^2 + C\|(a,b)\|_{H^s}\sum_{0 \leq \alpha + \beta \leq s}\|\nabla^\alpha\nabla^\beta_v \oper f\|_{L^2}^2\cr
&\quad + \delta\|\nabla^k\nabla^l_v \oper f\|_{L^2}^2\cr
&\leq C\sum_{0 \leq \alpha \leq s - l}\|\nabla^{\alpha + 1}\oper f\|_{L^2}^2 + C\epsilon_1\sum_{0 \leq \alpha + \beta \leq s}\|\nabla^\alpha\nabla^\beta_v \oper f\|_{L^2}^2\cr
&\quad + \delta\|\nabla^k\nabla^l_v \oper f\|_{L^2}^2,\cr
K_8 + K_9 &\leq C\|\nabla(a,b)\|_{H^{s-l}}^2 + C\|(a,b)\|_{H^s}\lt(\|\nabla (a,b)\|_{H^{s-1}}^2 + \|\nabla^k \nabla^l_v \oper f\|_{L^2}^2\rt)\cr
&\quad + \delta \|\nabla^k \nabla^l_v \oper f\|_{L^2}^2\cr
&\leq C\|\nabla(a,b)\|_{H^{s-l}}^2 + C\epsilon_1\lt(\|\nabla (a,b)\|_{H^{s-1}}^2 + \|\nabla^k \nabla^l_v \oper f\|_{L^2}^2\rt)+ \delta \|\nabla^k \nabla^l_v \oper f\|_{L^2}^2.
\end{aligned}$$
\end{proof}

\subsection{Macro-micro decomposition}\label{sec_mmd} 
In this part, we derive the hyperbolic-parabolic system which the macro components $a$ and $b$ satisfy. For this, we use the local conservation law of mass and the local balance law of momentum. More precisely, we multiply \eqref{main_eq} by $\sqrt{M}$ and $v_i\sqrt{M}$ for $1 \leq i \leq d$ and take the velocity integration over $\R^d$ to get
\begin{align}\label{mac_1}
\begin{aligned}
&\pa_t a + \nabla \cdot b = 0,\cr
&\pa_t b_i + \pa_i a + \sum_{j=1}^d \pa_j A_{ij}(\{ \bi - \bpp\}f) = 0.
\end{aligned}
\end{align}
Furthermore, we rewrite \eqref{main_eq} as
\bq\label{mac_2}
\pa_t \bpp f + v \cdot \nabla \bpp f + u_F \cdot \nabla_v \bpp f - \frac12 u_F \cdot v \bpp f - (u_F - b)\cdot v \sqrt{M} = -\pa_t \oper f + \ell + r,
\eq
where
$$\begin{aligned}
\ell &= - v \cdot \nabla \{\bi - \bpp \}f + \bl\{ \bi - \bpp\}f,\cr
r &= -\frac{b}{1 + a} \cdot \nabla_v \{ \bi - \bpp\} f + \frac{b}{2(1 + a)} \cdot \{ \bi - \bpp\}f.
\end{aligned}$$
Then we now apply the moment functional $A_{ij}(g) := \lal (v_i v_j - 1)\sqrt{M}, g \ral$ for $1 \leq i,j \leq d$ to \eqref{mac_2} to deduce
\bq\label{mac_3}
\pa_t A_{ij}(\{ \bi - \bpp\}f) + \pa_i b_j + \pa_j b_i - \frac{2b_i b_j}{1 + a} = A_{ij}(\ell + r).
\eq
The similar derivation of the system \eqref{mac_1}-\eqref{mac_3} is used for the study of collisional kinetic equations \cite{Duan1,Guo}. Using the system \eqref{mac_1}-\eqref{mac_3}, we find a temporal energy functional which has the dissipation rate $\|\nabla(a,b)\|_{H^{s-1}}^2$. Before we show it, we first provide the estimates of the term in the right hand side of the equation \eqref{mac_3} in the lemma below.
\begin{lemma}\label{lem_er}Let $d \geq 3$ and $s \geq 2[d/2] +2$, and let $T > 0$ be given. Suppose that $f\in \mc([0,T]; H^s(\R^d \times \R^d))$ is the solution to the equation \eqref{main_eq} satisfying
\[
\sup_{0 \leq t \leq T}\|f(t)\|_{H^s} \leq \epsilon_1 \ll 1.
\]
Then there exists a positive constant $C>0$ such that 
\bq\label{est_mac3_1}
\sum_{ 0 \leq k \leq s-1}\|\nabla^k A_{ij}(\ell)\|_{L^2} \leq C\sum_{0\leq k \leq s} \|\nabla^k \{ \bi - \bpp\}f\|_{L^2},
\eq
and
\bq\label{est_mac3_2}
\sum_{0 \leq k \leq s-1}\|\nabla^k A_{ij}(r)\|_{L^2} \leq C\epsilon_1\sum_{0 \leq k \leq s}\|\nabla^k \{ \bi - \bpp\}f\|_{L^2}.
\eq
\end{lemma}
\begin{proof} The estimate \eqref{est_mac3_1} can be obtained by using similar arguments as in \cite{CDM}. For the estimate \eqref{est_mac3_2}, we obtain
$$\begin{aligned}
\nabla^k A_{ij}(r) &=A_{ij}(\nabla^k r)\cr
&=\lt\lal (v_i v_j - 1)\sqrt{M}, \nabla^k \lt(\frac{b\cdot v}{2(1+a)}\{ \bi - \bpp\}f - \frac{b}{1+a}\cdot \nabla_v \{ \bi - \bpp\}f \rt) \rt\ral\cr
&=\sum_{0 \leq l\leq k} \binom{k}{l}\lt\lal (v_iv_j - 1)\sqrt{M}, \frac12\nabla^{k - l}\lt(\frac{b}{1+a}\rt)\cdot v\nabla^l \{ \bi - \bpp\}f \rt\ral\cr
&\quad - \sum_{0 \leq l\leq k} \binom{k}{l}\lt\lal (v_iv_j - 1)\sqrt{M}, \nabla^{k - l}\lt(\frac{b}{1+a}\rt)\cdot \nabla_v \nabla^l \{ \bi - \bpp\}f \rt\ral\cr
&\ls \sum_{0 \leq l \leq k}\lt\|\nabla^{k-l}\lt(\frac{b}{1+a}\rt)\rt\|_{L^2}\|\nabla^l \{ \bi - \bpp\}f\|_{L^2}.
\end{aligned}$$
This yields 
$$\begin{aligned}
\sum_{0\leq k \leq s-1}\|\nabla^k A_{ij}(r)\|_{L^2} &\leq C\lt\|\nabla\lt(\frac{b}{1+a} \rt) \rt\|_{H^{s-1}}\sum_{0 \leq k \leq s}\|\nabla^k \{ \bi - \bpp\}f\|_{L^2}\cr
&\leq C\epsilon_1\sum_{0 \leq k \leq s}\|\nabla^k \{ \bi - \bpp\}f\|_{L^2},
\end{aligned}$$
where we used the estimate in Lemma \ref{lem_us}.
\end{proof}
Inspired by Kawashima's hyperbolic-parabolic dissipation arguments \cite{Kawa}, we introduce the temporal energy functional $\mathcal{E}_0(f)$ by
\[
\mathcal{E}_0(f) := \sum_{0 \leq k \leq s-1}\sum_{1 \leq i,j \leq d}\int_{\R^d} \nabla^k (\pa_i b_j + \pa_j b_i)\nabla^k A_{ij}(\oper f)\,dx - \sum_{0\leq k \leq s-1}\int_{\R^d} \nabla^k a \nabla^k \nabla \cdot b\,dx.
\]

\begin{lemma}\label{lem_energy3}Let $d \geq 3$ and $s \geq 2[d/2] +2$, and let $T > 0$ be given. Suppose that $f\in \mc([0,T]; H^s(\R^d \times \R^d))$ is the solution to the equation \eqref{main_eq} satisfying
\[
\sup_{0 \leq t \leq T}\|f(t)\|_{H^s} \leq \epsilon_1 \ll 1.
\]
Then there exist positive constants $C_3, C > 0$ such that
\[
\frac{d}{dt}\mathcal{E}_0(f(t)) + C_3 \|\nabla (a,b)\|_{H^{s-1}}^2 \leq C\|\oper f\|_{L^2_v(H^s)}^2.
\]
\end{lemma}
\begin{proof} We first notice that
\[
\sum_{1 \leq i,j \leq d} \|\nabla^k (\pa_i b_j + \pa_j b_i)\|_{L^2}^2 = 2\|\nabla^{k+1}b\|_{L^2}^2 + 2\|\nabla \cdot \nabla^kb\|_{L^2}^2.
\]
On the other hand, it also follows from \eqref{mac_3} that
$$\begin{aligned}
\sum_{1 \leq i,j \leq d}\|\nabla^k (\pa_i b_j + \pa_j b_i)\|_{L^2}^2 &= -\frac{d}{dt}\sum_{1 \leq i,j \leq d}\int_{\R^d} \nabla^k (\pa_i b_j + \pa_j b_i)\nabla^k A_{ij}(\{ \bi - \bpp\}f)\,dx\cr
&\quad + \sum_{1 \leq i,j \leq d}\int_{\R^d} \nabla^k (\pa_i \pa_t b_j + \pa_j \pa_t b_i)\nabla^kA_{ij}(\{ \bi - \bpp\}f)\,dx\cr
&\quad + \sum_{1 \leq i,j \leq d}\int_{\R^d} \nabla^k (\pa_i b_j + \pa_j b_i)\nabla^k\lt( \frac{2b_ib_j}{1 + a}  + A_{ij}(\ell + r)\rt)dx\cr
&=: \sum_{i=1}^3 J_i,
\end{aligned}$$
where $J_i,i=2,3$ are estimated as follows.
$$\begin{aligned}
J_2 &= 2\sum_{1 \leq i,j \leq d}\int_{\R^d} \nabla^k \Big(\pa_i a + \sum_{1\leq l \leq d}\pa_l A_{il}(\{ \bi - \bpp\}f) \Big) \nabla^k \pa_j A_{ij}(\{ \bi - \bpp\}f)\,dx\cr
&\leq \epsilon_1\|\nabla a\|_{H^{s-1}}^2 + C\|\nabla \{ \bi - \bpp\}f\|_{L^2_v(H^{s-1})}^2,\cr
J_3 &\leq \frac12 \sum_{1 \leq i,j \leq d}\|\nabla^k (\pa_i b_j + \pa_j b_i)\|_{L^2}^2 + C\sum_{1 \leq i,j \leq d}\lt\| \frac{2b_i b_j}{1+a}\rt\|_{H^{s-1}}^2 + C\|\{ \bi - \bpp\}f\|_{L^2_v(H^s)}^2\cr
&\leq \frac12 \sum_{1 \leq i,j \leq d}\|\nabla^k (\pa_i b_j + \pa_j b_i)\|_{L^2}^2 + C\epsilon_1\|\nabla b\|_{H^{s-1}}^2 + C\|\{ \bi - \bpp\}f\|_{L^2_v(H^s)}^2.
\end{aligned}$$
Here, we used the equation $\eqref{mac_1}_2$ for $J_2$ and the estimates in Lemma \ref{lem_er} for $J_3$.
Thus, we obtain
$$\begin{aligned}
&\frac{d}{dt}\sum_{0 \leq k \leq  s-1}\sum_{1 \leq i, j \leq d}\int_{\R^d} \nabla^k (\pa_i b_j + \pa_j b_i) \nabla^k A_{ij}(\oper f)\,dx + \|\nabla b\|_{H^{s-1}}^2 + \|\nabla \cdot  b\|_{H^{s-1}}^2\cr
&\quad \leq \epsilon_1\|\nabla a\|_{H^{s-1}}^2 + C\|\oper f\|_{L^2_v(H^s)}^2 + C\epsilon_1 \|\nabla b\|_{H^{s-1}}^2.
\end{aligned}$$
We again use $\eqref{mac_1}$ to get the dissipation rate of $\|\nabla a\|_{H^{s-1}}$. By taking $\nabla^k$ to $\eqref{mac_1}_2$ and multiplying it by $\nabla^k \pa_i a$, we find
$$\begin{aligned}
\|\nabla^{k+1} a\|_{L^2}^2 &= \frac{d}{dt} \int_{\R^d} \nabla^k a \cdot \nabla^k \nabla \cdot b\,dx + \sum_{1 \leq i \leq d} \int_{\R^d} \nabla^k \pa_i \pa_t a \nabla^k b_i\,dx\cr
&\quad - \sum_{1 \leq i,j \leq d} \int_{\R^d} \nabla^k \pa_i a \cdot \pa_j \nabla^k A_{ij}(\oper f)\,dx\cr
&\leq \frac{d}{dt} \int_{\R^d} \nabla^k a \cdot \nabla^k \nabla \cdot b\,dx + \frac12 \|\nabla^{k+1} a\|_{L^2}^2 + \|\nabla^k \nabla \cdot b\|_{L^2}^2 + C\|\nabla \oper f\|_{L^2_v(H^{s-1})}^2,
\end{aligned}$$
due to the local conservation law of mass $\eqref{mac_1}_1$. This implies
\[
-\frac{d}{dt}\int_{\R^d} \nabla^k a \cdot \nabla^k \nabla \cdot b\,dx + \frac12 \|\nabla^{k+1} a\|_{L^2}^2 \leq  \|\nabla^k \nabla \cdot b\|_{L^2}^2 + C\|\nabla \oper f\|_{L^2_v(H^{s-1})}^2.
\]
Now we combine all the estimates above to have
\[
\frac{d}{dt}\mathcal{E}_0(f) + \|\nabla b\|_{H^{s-1}}^2 + \lt(\frac12 - \epsilon_1\rt)\|\nabla a\|_{H^{s-1}}^2 \leq C\epsilon_1\|\nabla b\|_{H^{s-1}}^2 + C\|\oper f\|_{L^2_v(H^s)}^2.
\]
\end{proof}
\begin{remark}\label{rmk_energy3} Using the definition of $\mathcal{E}_0(f)$, we find 
$$\begin{aligned}
\mathcal{E}_0(f) &\ls \sum_{0 \leq k \leq s}\lt( \|\nabla^k \oper f\|_{L^2}^2 + \|\nabla^k (a,b)\|_{L^2}^2\rt)\cr
&\ls \sum_{0 \leq k \leq s}\lt( \|\nabla^k \oper f\|_{L^2}^2 + \|\nabla^k \bpp f\|_{L^2}^2\rt)\cr
&\ls \|f\|_{L^2_v(H^s)}^2,
\end{aligned}$$
i.e., the temporal energy functional $\mathcal{E}_0(f(t))$ is bounded by $\|f(t)\|_{L^2_v(H^s)}^2$ for all $t \in [0,T]$.
\end{remark}
%
%
%
%
\section{Proof of Theorem \ref{thm_main1}}\label{sec_ext}

%
%
%
%

\subsection{Existence and uniqueness}\label{ssec_ext}In this subsection, we provide the details of proof for the part of global existence and uniqueness of classical solutions in Theorem \ref{thm_main1}.

We first combine the estimates in Lemmas \ref{lem_energy1} and \ref{lem_energy3} to get
\[
\frac{d}{dt} \lt( \nu_1 \|f\|_{L^2_v(H^s)}^2 + \mathcal{E}_0(f)\rt) + (C_1 \nu_1 - C)\sum_{0 \leq k \leq s}\|\nabla^k \oper f\|_\mu^2 + (C_3 - C\epsilon_1 \nu_1)\|\nabla(a,b)\|_{H^{s-1}}^2 \leq 0,
\]
for some large positive constant $\nu_1 \geq 1$. Note that for some $\nu_1 > 0$ large enough it is clear to get $\nu_1 \|f\|_{L^2_v(H^s)}^2 + \mathcal{E}_0(f) \approx \|f\|_{L^2_v(H^s)}^2$ in the sense that there exists a positive constant $C > 0$ such that
\[
\frac{1}{C}\|f\|_{L^2_v(H^s)}^2 \leq \nu_1 \|f\|_{L^2_v(H^s)}^2 + \mathcal{E}_0(f) \leq C\|f\|_{L^2_v(H^s)}^2,
\]
due to Remark \ref{rmk_energy3}. It also follows from the linear combination of \eqref{energy2} over $1 \leq l \leq s$ that
$$\begin{aligned}
&\frac{d}{dt}\sum_{1 \leq l \leq s}C^l \sum_{0 \leq k \leq s-l}\|\nabla^k \nabla^l_v \oper f\|_{L^2}^2 + C_4\sum_{1 \leq l \leq s}\sum_{0 \leq k \leq s - l}\|\nabla^k \nabla^l_v \oper f\|_\mu^2\cr
&\quad \leq C\epsilon_1\lt( \sum_{0 \leq k + l \leq s} \|\nabla^k\nabla^l_v \oper f\|_\mu^2 + \|\nabla (a,b)\|_{H^{s-1}}^2\rt) \cr
&\qquad + C\sum_{0 \leq k \leq s}\|\nabla^k \oper f\|_\mu^2 + C\|\nabla (a,b)\|_{H^{s-1}}^2,
\end{aligned}$$
for some positive constants $C, C_4, C^l > 0, ~l = 1,\cdots, s$. We now set a total energy functional ${\mathcal{E}}(f)$ and a dissipation rate ${\mathcal{D}}(f)$:
\begin{align}\label{not_ed}
\begin{aligned}
{\mathcal{E}}(f) &:= \nu_2\lt(\nu_1 \|f\|_{L^2_v(H^s)}^2 + \mathcal{E}_0(f)\rt) + \sum_{1 \leq l \leq s}C^l \sum_{0 \leq k \leq s-l}\|\nabla^k \nabla^l_v \oper f\|_{L^2}^2,\cr
{\mathcal{D}}(f) &:= \sum_{0 \leq k + l \leq s} \|\nabla^k\nabla^l_v \oper f\|_\mu^2 + \|\nabla (a,b)\|_{H^{s-1}}^2,
\end{aligned}
\end{align}
where $\nu_2 \geq 1$ is a sufficiently large positive constant. 
Then we obtain
\[
\frac{d}{dt}{\mathcal{E}}(f(t)) + C_5{\mathcal{D}}(f(t)) \leq 0,
\]
where $C_5 > 0$ is a positive constant independent of $T$ and $f_0$, and this yields
\[
\sup_{0 \leq t \leq T}\lt\{{\mathcal{E}}(f(t)) + C_5\int_0^t {\mathcal{D}}(f(\tau))\,d\tau\rt\} \leq {\mathcal{E}}(f_0).
\]
On the other hand, since $\|\nabla^k \nabla^l_v \bpp f\|_{L^2} \approx \|\nabla^k f\|_{L^2}$, we find
\begin{align}\label{eqv_f}
\begin{aligned}
\|f\|_{H^s} &= \sum_{0 \leq k + l \leq s}\|\nabla^k \nabla^l_v \lt( \bpp f + \oper f\rt)\|_{L^2}\cr
& \approx \|f\|_{L^2} + \sum_{1 \leq k \leq s}\|\nabla^k f\|_{L^2} + \sum_{1 \leq l \leq s} \sum_{0 \leq k \leq s - l}\|\nabla^k \nabla^l_v \oper f\|_{L^2},
\end{aligned}
\end{align}
and this implies $\mathcal{E}(f) \approx \|f\|_{H^s}^2$. Hence we have
\bq\label{est_f}
\sup_{0 \leq t \leq T}\lt\{ \|f(t)\|_{H^s}^2 + C\int_0^t {\mathcal{D}}(f(\tau))\,d\tau\rt\} \leq C_0\|f_0\|_{H^s}^2,
\eq
where $C, C_0 > 0$ are positive constants independent of $T$ and $f_0$. We now 
choose a positive constant 
\[
L =\min \{\epsilon_0, \epsilon_1\},
\]
where $\epsilon_0$ and $\epsilon_1$ are appeared in Theorem \ref{thm_loc} and in Section \ref{sec_apri}, respectively. We also choose the initial data $f_0$ satisfying
\[
M + \sqrt{M}f_0 \geq 0 \quad \mbox{and} \quad \|f_0\|_{H^s} \leq \frac{L}{2\sqrt{1 + C_0}}.
\]
Define the lifespan of solutions to the system \eqref{main_eq} by
\[
T := \sup\lt\{t : \sup_{0 \leq \tau \leq t}\|f(\tau)\|_{H^s}^2 \rt\}.
\]
Since 
\[
\|f_0\|_{H^s} \leq \frac{L}{2\sqrt{1 + C_0}} \leq \frac L2 \leq \epsilon_0,
\]
it follows from Theorem \ref{thm_loc} and the continuation argument that 
the lifespan $T > 0$ is positive. If $T < + \infty$, we can deduce from the definition of $T$ that
\bq\label{contra}
\sup_{0 \leq \tau \leq T}\|f(\tau)\|_{H^s} = L.
\eq
On the other hand, it follows from \eqref{est_f} that
\[
\sup_{0 \leq \tau \leq T}\|f(\tau)\|_{H^s} \leq \sqrt{C_0}\|f_0\|_{H^s} \leq \frac{L\sqrt{C_0}}{2\sqrt{1 + C_0}} \leq \frac L2, 
\]
and this is a contraction to \eqref{contra}. Therefore $T = +\infty$, and this concludes that the local solution $f$ obtained in Theorem \ref{thm_loc} can be extended to the infinite time. 
%
%
%
%

\subsection{Large-time behavior}
In this part, we show the time-decay estimate of solutions obtained in Section \ref{ssec_ext}, and complete the proof of Theorem \ref{thm_main1}. We also remark the exponential decay rate of convergence when the spatial periodic domain is considered.

We first show the estimate of time decay for the solutions $\|f\|_{L^2}$ in terms of the total energy function $\mathcal{E}(f)$ given in \eqref{not_ed}.
\begin{lemma}If $\|f_0\|_{L^2_v(L^1)}$ is bounded, we have
\begin{align}\label{est_lem_lt}
\begin{aligned}
\|f(t)\|_{L^2}^2 &\leq C\lt({\mathcal{E}}(f_0) + \|f_0\|_{L^2_v(L^1)} \rt)(1 + t)^{-d/2} \cr
&\quad + C\int_0^t (1 + t - \tau)^{-d/2}{\mathcal{E}}(f(\tau))^2\,d\tau + C\lt( \int_0^t (1+t-\tau)^{-d/4}{\mathcal{E}}(f(\tau))\,d\tau\rt)^2,
\end{aligned}
\end{align}
for $t \geq 0$.
\end{lemma}
\begin{proof}We rewrite the equation \eqref{main_eq} in the mild form: 
\[
f(t) = e^{\bb t}f_0 + \int_0^t e^{\bb(t-\tau)}\lt(H_1(f(\tau)) + H_2(f(\tau))\rt)\,d\tau,
\]
where
$$\begin{aligned}
H_1(f) &:= \frac12 u_F \cdot v \oper f - u_F \cdot \nabla_v \oper f,\cr H_2(f) &:= (u_F - b)\cdot v \sqrt{M} + \frac12 u_F \cdot v \bpp f - u_F \cdot \nabla_v \bpp f.
\end{aligned}$$
Then it follows from Proposition \ref{prop_hypo} that
\begin{align}\label{est_lt1}
\begin{aligned}
\|f(t)\|_{L^2}^2 &\leq C\lt({\mathcal{E}}(f_0) + \|f_0\|_{L^2_v(L^1)}^2\rt)(1 + t)^{-d/2} \cr
&\quad + C\int_0^t (1 + t - \tau)^{-d/2}\lt(\|\mu^{-1/2}H_1(f(\tau))\|_{L^2_v(L^1)}^2 + \|\mu^{-1/2}H_1(f(\tau))\|_{L^2}^2 \rt)d\tau\cr
&\quad + C\lt( \int_0^t (1 + t - \tau)^{-d/4}\lt( \|H_2(f(\tau))\|_{L^2_v(L^1)} + \|H_2(f(\tau))\|_{L^2}\rt)d\tau\rt)^2,
\end{aligned}
\end{align}
since $H_1$ satisfies the condition \eqref{condi_h}.
Note that 
$$\begin{aligned}
\|(u_F - b)\cdot v\sqrt{M}\|_{L^2_v(L^1)} + \|(u_F - b)\cdot v\sqrt{M}\|_{L^2} &\leq C\|u_F - b\|_{L^1} + C\|u_F - b\|_{L^2}\cr
&\leq C\lt(1 + \|a\|_{L^2} \rt)\|b\|_{L^2}\cr
&\leq C{\mathcal{E}}(f).
\end{aligned}$$
This and together with similar estimates  as in \cite{DFT}, we obtain
\begin{align}\label{est_lt2}
\begin{aligned}
&\|\mu^{-1/2}H_1(f(\tau))\|_{L^2_v(L^1)}^2 + \|\mu^{-1/2}H_1(f(\tau))\|_{L^2}^2 \leq C{\mathcal{E}}(f(\tau))^2,\cr
&\|H_2(f(\tau))\|_{L^2_v(L^1)} + \|H_2(f(\tau))\|_{L^2} \leq C{\mathcal{E}}(f(\tau)).
\end{aligned}
\end{align}
Finally, we combine \eqref{est_lt1} and \eqref{est_lt2} to conclude the desired result.
\end{proof}
We set
\[
\mathcal{E}_\infty(t) := \sup_{0 \leq \tau \leq t} (1 + \tau)^{d/2}{\mathcal{E}}(f(\tau)).
\]
Then it follows from \eqref{est_lem_lt} that
$$\begin{aligned}
\|f(t)\|_{L^2}^2 &\leq C\lt({\mathcal{E}}(f_0) + \|f_0\|_{L^2_v(L^1)}^2 \rt)\lt( 1 + t\rt)^{-d/2} + C\|f_0\|_{H^s}^2\int_0^t (1 + t - \tau)^{-d/2} {\mathcal{E}}(f(\tau))\,d\tau\cr
&\quad + C\|f_0\|_{H^s}^{1/3}\lt(\int_0^t (1 + t - \tau)^{-d/4} {\mathcal{E}}(f(\tau))^{11/12}\,d\tau\rt)^2\cr
&\quad =:\sum_{i=1}^3 I_i,
\end{aligned}$$
due to ${\mathcal{E}}(f) \leq {\mathcal{E}}(f_0) \leq C\|f_0\|_{H^s}^2$. We also use Lemma \ref{lem_ei} to deduce
\[
I_2 \leq C\|f_0\|_{H^s}^2\mathcal{E}_\infty(t)\int_0^t (1 + t - \tau)^{-d/2}(1 + \tau)^{-d/2}\,d\tau \leq C\|f_0\|_{H^s}^2\mathcal{E}_\infty(t)(1 + t)^{-d/2}.
\]
Similarly, we use the fact $11d /24 > \max\{1, d/4\}$ to obtain
\[
I_3 \leq C\|f_0\|_{H^s}^{1/3}\mathcal{E}_\infty(t)^{11/6}\lt(\int_0^t (1 + t - \tau)^{-d/4}(1+\tau)^{-11d/24}\,d\tau\rt)^2 \leq C\|f_0\|_{H^s}^{1/3}\mathcal{E}_\infty(t)^{11/6}(1 + t)^{-d/2}.
\]
Thus we have
\[
\|f(t)\|_{L^2}^2 \leq C\lt(N_0 + \|f_0\|_{H^s}^2\mathcal{E}_\infty(t) + \|f_0\|_{H^s}^{1/3}\mathcal{E}_\infty(t)^{11/6}\rt)(1 + t)^{-d/2},
\]
where $N_0 := \|f_0\|_{H^s}^2 + \|f_0\|_{L^2_v(L^1)}^2 > 0$,
and this yields
\[
\mathcal{E}_\infty(t) \leq C\lt(N_0 + \|f_0\|_{H^s}\mathcal{E}_\infty(t) + \|f_0\|_{H^s}^{1/6}\mathcal{E}_\infty(t)^{11/6}\rt).
\]
Since $\|f_0\|_{H^s} \ll 1$, this concludes the uniform boundedness of $\mathcal{E}_\infty (t)$. Hence, we have the algebraic decay of the solutions to the equation \eqref{main_eq}.
\begin{remark}\label{rmk_lt}If we consider the periodic spatial domain $\T^d$, we find the following differential inequality using similar arguments as in Section \ref{sec_apri}.
\[
\frac{d}{dt}\mathcal{E}(f(t)) + C\mathcal{D}(f(t)) \leq 0,
\]
where $\mathcal{E}(f)$ and $\mathcal{D}(f)$ are given in \eqref{not_ed}. It also follows from \eqref{eqv_f} that $\mathcal{E}(f) \approx \|f\|_{H^s}^2$. We now further assume 
\[
\int_{\T^d} a_0(x) \,dx = 0 \quad \mbox{and} \quad \int_{\T^d} b_0(x)\,dx = 0,
\]
then this implies 
\[
\int_{\T^d} a(x,t) \,dx = 0 \quad \mbox{and} \quad \int_{\T^d} b(x,t)\,dx = 0, \quad \mbox{for} \quad t \geq 0,
\]
due to the conservations of mass and momentum. Thus, by using the Poincar\'e inequality, we get
\[
\|(a,b)\|_{L^2} \leq C\|\nabla (a,b)\|_{L^2}.
\]
This deduces
\[
\sum_{0 \leq k + l \leq s}\|\nabla^k \nabla^l_v \bpp f\|_{L^2} \ls \sum_{0 \leq k \leq s} \|\nabla^k (a,b)\|_{L^2} \ls \sum_{1 \leq k \leq s}\|\nabla^k(a,b)\|_{L^2} \ls \|\nabla(a,b)\|_{H^{s-1}}.
\]
Hence we have
$$\begin{aligned}
\|f\|_{H^s} &\leq \sum_{0 \leq k + l \leq s}\|\nabla^k \nabla^l_v \bpp f\|_{L^2}+ \sum_{0 \leq k + l \leq s}\|\nabla^k \nabla^l_v \oper f\|_{L^2}\cr
&\ls \sqrt{\mathcal{D}(f)},
\end{aligned}$$
i.e., $\mathcal{E}(f) \ls \mathcal{D}(f)$. This concludes 
\[
\frac{d}{dt}\mathcal{E}(f(t)) + C\mathcal{E}(f(t)) \leq 0,
\]
and 
\[
\|f(t)\|_{H^s}^2 \leq C \mathcal{E}(f(t)) \leq C\mathcal{E}(f_0)e^{-Ct} \leq C\|f_0\|_{H^s}^2 e^{-Ct} \quad \mbox{for} \quad t \geq 0.
\]
\end{remark}
%
%
%
%

\appendix

\section{Proof of Proposition \ref{prop_hypo}}\label{app_hypo}
We consider the following Cauchy problem:
\bq\label{app_eq}
\pa_t f + v \cdot \nabla f - b \cdot v \sqrt{M} = \bl f + h,
\eq
where $h$ is a given function satisfying the conditions in Proposition \ref{prop_hypo}.
For an integrable function $g : \R^d \to \R$, we define its Fourier transform $\widehat g := \mathcal{F} g$ by
\[
\widehat{g}(k) = \mathcal{F} g(k) := \int_{\R^d} e^{-2\pi \mi x \cdot k} g(x)\,dx, 
\]
where $k \in \R^d$ and $\mi \in \sqrt{-1} \in \mathbb{C}$ is the imaginary unit. We denote the dot product $a \cdot \bar b$ for $a, b \in \mathbb{C}^d$ by $(a\,|\, b)$, where $\bar b$ is the complex conjugate of $b$. Then we first show the $L^2$-estimate of $\widehat f$ in the following lemma.
\begin{lemma} Let $f$ be a solution of the equation \eqref{app_eq} satisfying the conditions in Proposition \ref{prop_hypo}. Then we have
\[
\frac{\pa}{\pa t}\|\wh f\|_{L^2_v}^2 + \lambda|\oper \wh f|_\mu^2 \leq C\|\mu^{-1/2} \wh h\|_{L^2_v}^2.
\]
\end{lemma}
\begin{proof}The Fourier transform of \eqref{app_eq} gives
\[
\pa_t \wh f + \mi v \cdot k \wh f - \wh b \cdot v \sqrt{M} = \bl \wh f + \wh h = \bl \oper \wh f - \wh b \cdot v \sqrt{M} + \wh h.
\]
Then we use $\{ \bi - \bpp_0\}\oper = \oper $ to get
\[
\frac12 \frac{\pa}{\pa t}\|\wh f\|_{L^2_v}^2 + \lambda|\oper \wh f|_\mu^2 \leq |(\wh h\,|\,\wh f)|.
\]
On the other hand, we obtain
$$\begin{aligned}
|(\wh h\,|\,\wh f)| &= |(\wh h\,|\, \oper \wh f)|\cr
&\leq \delta\|\mu^{1/2}\oper \wh f\|_{L^2_v}^2 + \frac{1}{4\delta}\|\mu^{-1/2}\wh f\|_{L^2_v}^2\cr
&\leq C\delta|\oper \wh f|_{\mu}^2 + \frac{1}{4\delta}\|\mu^{-1/2}\wh f\|_{L^2_v}^2,
\end{aligned}$$
for any positive constant $\delta > 0$. We now choose the $\delta > 0$ such that $C\delta < \lambda/2$ to complete the proof.
\end{proof}
Similar as in Section \ref{sec_mmd}, we derive the macroscopic balance laws:
\begin{align}\label{est_a_1}
\begin{aligned}
&\pa_t a + \nabla \cdot b = 0,\cr
&\pa_t b_i + \pa_i a + \sum_{j=1}^d \pa_j A_{ij}(\{ \bi - \bpp\}f) = 0,\cr
&\pa_t A_{ij}(\{ \bi - \bpp\}f) + \pa_i b_j + \pa_j b_i = A_{ij}(\ell + h),
\end{aligned}
\end{align}
where 
\[
\ell = - v \cdot \nabla \{\bi - \bpp \}f + \bl\{ \bi - \bpp\}f.
\]
Then it follows from \eqref{est_a_1} that for $1 \leq m \leq d$
\begin{align}\label{est_a_2}
\begin{aligned}
-\Delta b_m &- \pa_t \lt( \sum_{1 \leq j \leq d} \pa_j A_{jm}(\oper f) - \frac12 \sum_{1 \leq j \leq m}\pa_m A_{jj}(\oper f)\rt)\cr
& = \frac12 \sum_{1 \leq j \leq d}\pa_mA_{jj}(\ell + h) - \sum_{1\leq j \leq d}\pa_j A_{jm}(\ell + h).
\end{aligned}
\end{align}
By taking the Fourier transform to \eqref{est_a_2}, we deduce
$$\begin{aligned}
|k|^2 \wh b_m &- \pa_t \lt( \sum_{1 \leq j \leq d} \mi k_j A_{jm}(\oper \wh f) - \frac12 \sum_{1 \leq j \leq m}\mi k_m A_{jj}(\oper \wh f)\rt)\cr
& = \frac12 \sum_{1 \leq j \leq d}\mi k_m A_{jj}(\wh \ell + \wh h) - \sum_{1\leq j \leq d}\mi k_j A_{jm}(\wh \ell + \wh h).
\end{aligned}$$
This yields 
$$\begin{aligned}
&\pa_t \lt( \sum_{1 \leq j \leq d} \mi k_j A_{jm}(\oper \wh f) - \frac12 \sum_{1 \leq j \leq m}\mi k_m A_{jj}(\oper \wh f) \,\Big|\,\wh b_m \rt) + |k|^2 |\wh b_m|^2\cr
&\quad = \lt( \frac12 \sum_{1 \leq j \leq d} \mi k_m A_{jj}(\wh \ell + \wh h) - \sum_{1 \leq j \leq d}\mi k_jA_{jm}(\wh \ell + \wh h) \,\Big|\, \wh b_m\rt)\cr
&\qquad + \lt(\sum_{1 \leq j \leq d} \mi k_j A_{jm}(\oper \wh f) - \frac12 \sum_{1 \leq j \leq m}\mi k_m A_{jj}(\oper \wh f) \,\Big|\, \pa_t \wh b_m \rt)\cr
&\quad =: I_1 + I_2,
\end{aligned}$$
where $I_1$ is estimated as follows.
$$\begin{aligned}
I_1 &\leq \frac12 |k|^2|\wh b_m|^2 + C\sum_{1 \leq i,j \leq d}\lt( |A_{ij}(\wh \ell)|^2 + |A_{ij}(\wh h)|^2\rt) \cr
&\leq \frac12 |k|^2|\wh b_m|^2 + C(1 + |k|^2)\|\oper \wh f\|_{L^2_v}^2 + C\|\mu^{-1/2}\wh h\|_{L^2_v}^2.
\end{aligned}$$
For the estimate of $I_2$, we use the following Fourier transform of $\eqref{est_a_1}_2$:
\bq\label{est_a_3}
\pa_t \wh b_m + \mi k_m \wh a + \sum_{1 \leq j \leq d}\mi k_j A_{mj}(\oper \wh f) = 0.
\eq
Then we deduce
$$\begin{aligned}
I_2 &= \lt(\sum_{1 \leq j \leq d} \mi k_j A_{jm}(\oper \wh f) - \frac12 \sum_{1 \leq j \leq m}\mi k_m A_{jj}(\oper \wh f) \,\Big|\, - \mi k_m \wh a - \sum_{1 \leq j \leq d}\mi k_j A_{mj}(\oper \wh f)  \rt)\cr
&\leq \delta|k|^2 |\wh a|^2 + C(1+|k|^2)\| \oper \wh f\|_{L^2_v}^2.
\end{aligned}$$
Thus we have
\begin{align}\label{est_a_e1}
\begin{aligned}
&\pa_t \lt( \sum_{1 \leq j \leq d} \mi k_j A_{jm}(\oper \wh f) - \frac12 \sum_{1 \leq j \leq m}\mi k_m A_{jj}(\oper \wh f) \,\Big|\,\wh b_m \rt) + |k|^2 |\wh b_m|^2\cr
&\quad \leq \delta|k|^2 |\wh a|^2 + C(1+|k|^2)\| \oper \wh f\|_{L^2_v}^2 + C\|\mu^{-1/2}\wh h\|_{L^2_v}^2.
\end{aligned}
\end{align}
We next obtain the dissipation of $|k|^2|\wh a|^2$. For this, we find from \eqref{est_a_3} that
\[
\lt(-\pa_t \mi k \cdot \wh b \,\big|\, \wh a\rt) + |k|^2|\wh a|^2 + \lt( \sum_{1 \leq i,j\leq d} k_ik_j A_{ij}(\oper \wh f) \,\Big|\, \wh a \rt) = 0.
\]
Moreover, we get
\[
\lt(-\pa_t \mi k \cdot \wh b \,\big|\, \wh a\rt) = \pa_t\lt( - \mi k \cdot \wh b \,\big|\, \wh a\rt) + \lt( \mi k \cdot \wh b \,\big|\, \pa_t \wh a\rt) = \pa_t\lt( - \mi k \cdot \wh b \,\big|\, \wh a\rt) - |k \cdot \wh b|^2.
\]
Using the above equation, we have
\bq\label{est_a_e2}
\pa_t \mbox{Re}\lt( - \mi k \cdot \wh b \,\big|\, \wh a\rt) + \frac12 |k|^2 |\wh a|^2 \leq |k|^2|\wh b|^2 + C|k|^2\|\oper \wh f\|_{L^2_v}^2.
\eq
Set
$$\begin{aligned}
E(\wh f) &:= 3\sum_{1 \leq m, j \leq d}\frac{\mi k_j}{1 + |k|^2}\lt(A_{jm}(\oper \wh f) \,\big| \, \wh b_m \rt) - \frac32 \sum_{1 \leq m,j \leq d}\frac{\mi k_m}{1 + |k|^2}\lt(A_{jj}(\oper \wh f) \,\big| \, \wh b_m \rt)\cr
&\quad - \frac{\mi k}{1 + |k|^2} \cdot \lt(\wh b \,\big| \, \wh a \rt).
\end{aligned}$$
Combining \eqref{est_a_e1} and \eqref{est_a_e2}, we find that $E(\wh f)$ satisfies
\[
\frac{\pa}{\pa t} \mbox{Re} E(\wh f) + \frac{|k|^2}{4(1 + |k|^2)} \lt( |\wh a|^2 + |\wh b|^2\rt) \leq  C\|\oper \wh f\|_{L^2_v}^2 + C\|\mu^{-1/2}\wh h\|_{L^2_v}^2.
\]
We next set 
\[
\widetilde{E}(\wh f) := \|\wh f\|_{L^2_v}^2 + \kappa \mbox{Re} E(\wh f),
\]
where $\kappa > 0$ is a positive constant which will determined later. Then we obtain
\[
\pa_t \widetilde{E}(\wh f) + \lt( \lambda - C\kappa\rt)\|\oper \wh f\|_{L^2_v}^2 + \frac{|k|^2}{4(1 + |k|^2)} \lt( |\wh a|^2 + |\wh b|^2\rt) \leq C\|\mu^{-1/2}\wh h\|_{L^2_v}^2.
\]
Since
\[
\|\oper \wh f\|_{L^2_v}^2 + \frac{|k|^2}{1 + |k|^2} \lt( |\wh a|^2 + |\wh b|^2\rt) \geq \frac{|k|^2}{1 + |k|^2}\lt(\|\oper \wh f\|_{L^2_v}^2 + |\wh a|^2 + |\wh b|^2\rt)
\]
and
\[
\widetilde{E}(\hat f) \leq C\lt(\|\oper \wh f\|_{L^2_v}^2 + |\wh a|^2 + |\wh b|^2\rt),
\]
we have 
\[
\frac{d}{dt}\widetilde{E}(\wh f) + \frac{C_0|k|^2}{1 + |k|^2}  \widetilde{E}(\wh f) \leq C\|\mu^{-1/2}\wh h\|_{L^2_v}^2,
\]
for some positive constant $C_0 > 0$.
This implies
\[
\widetilde{E}(\wh f(k,t)) \leq \widetilde{E}(\wh f(0,k))e^{-\frac{C_0|k|^2}{1+|k|^2}t} + C\int_0^t e^{-\frac{C_0|k|^2}{1 + |k|^2}(t-\tau)}\|\mu^{-1/2}\oper \wh f(\tau)\|_{L^2_v}^2\,d\tau.
\]
Let $h = 0$, then $f(t) = e^{\bb t}f_0$ is the solution to \eqref{app_eq}  and
$$\begin{aligned}
\|\nabla^\alpha e^{\bb t}f_0\|_{L^2}^2 &= \int_{\R^d} |k^{2\alpha}| |\widetilde{E}(\wh f(k,t))|\,dk\cr
&\leq \int_{\R^d} |k^{2\alpha}||e^{-\frac{C_0|k|^2}{1 + |k|^2}t}\|\wh f_0(k)\|_{L^2_v}^2\,dk\cr
&\leq \int_{|k| \leq 1}|k^{2(\alpha - \beta)}|e^{-\frac{C_0|k|^2}{1 + |k|^2}t}|k^{2\beta}|\|\wh f_0(k)\|_{L^2_v}^2\,dk + \int_{|k|\geq 1} e^{-\frac{C_0}{2}t}|k^{2\alpha}|\|\wh f_0(k)\|_{L^2_v}^2\,dk\cr
&\leq C(1 + t)^{-\frac{d}{q} + \frac{d - 2|\alpha - \beta|}{2}}\|\nabla^\beta f_0\|_{L^2_v(L^q)}^2 + Ce^{-\frac{C_0}{2}t}\|\nabla^k f_0\|_{L^2}^2,
\end{aligned}$$
where $k^\alpha = k_1^{\alpha_1}k_2^{\alpha_2}\cdots k_d^{\alpha_d}$. Here H$\ddot{\textrm{o}}$lder and Hausdorff-Young inequalities are used as in \cite{Kawa}. We conclude the desired result by using the similar techniques as the above.

%
%
%
%
 
\section*{Acknowledgments}

The author was supported by Engineering and Physical Sciences Research Council(EP/K008404/1). The author also acknowledges the support of the ERC-Starting Grant HDSPCONTR ``High-Dimensional Sparse Optimal Control''.

%
%
%
%


\begin{thebibliography}{10}
\bibitem{BDGM} L. Boudin, L. Desvillettes, C. Grandmont, and A. Moussa, Global existence of solutions for the coupled Vlasov and Navier-Stokes equations, Diff. Int. Eqns, 22, (2009), 1247--1271.
\bibitem{CCK} J. A. Carrillo, Y.-P. Choi, and T. Karper, On the analysis of a coupled kinetic-fluid model with local alignment forces, Ann. I. H. Poincar\'e - AN, 33, (2016), 273--307.
\bibitem{CCTT} J. A. Carrillo, Y.-P. Choi, E. Tadmor, and C. Tan, Critical thresholds in 1D Euler equations with nonlocal forces, Math. Mod. Meth. Appl. Sci., 26, (2016), 185--206.
\bibitem{CDM} J. A. Carrillo, R. Duan, and A. Moussa, Global classical solutions close to the equilibrium to the Vlasov-Fokker-Planck-Euler system, Kinetic and Related Models, 4, (2011), 227--258.
\bibitem{CFRT} J. A. Carrillo, M. Fornasier, J. Rosado, and G. Toscani, Asymptotic flocking dynamics for the kinetic Cucker-Smale model, SIAM J. Math. Anal., 42, (2010), 218--236.
\bibitem{Choi} Y.-P. Choi, Compressible Euler equations interacting with incompressible flow, Kinetic and Related Models, 8, (2015), 335--358.
\bibitem{CHL} Y.-P. Choi, S.-Y. Ha, and Z. Li, Emergent dynamics of the Cucker-Smale flocking model and its variants, preprint (2016).
\bibitem{CS} F. Cucker and S. Smale, Emergent behavior in flocks, IEEE Trans. Auto. Control, 52, (2007), 852--862.
\bibitem{Duan1} R. Duan, On the Cauchy problem for the Boltzmann equation in the whole space: Global existence and uniform stability in $L^2_\xi(H^N_x)$, J. Diff. Eqns., 244, (2007), 3204--3234.
\bibitem{Duan} R. Duan, Hypocoercivity of linear degenerately dissipative kinetic equations, Nonlinearity, 24, (2011), 2165--2189.
\bibitem{DFT} R. Duan, M. Fornasier, and G. Toscani, A kinetic flocking model with diffusion, Comm. Math. Phys., 300, (2010), 95--145.
\bibitem{Frank} T. D. Frank, Nonlinear Fokker-Planck Equations: Fundamentals and Applications, Springer-Verlag Berlin Heidelberg, (2005).
\bibitem{Guo} Y. Guo, The Boltzmann equation in the whole space, Indiana Univ. Math. J., 53, (2004), 1081--1094. 
\bibitem{HT} S.-Y. Ha and E. Tadmor, From particle to kinetic and hydrodynamic descriptions of flocking, Kinetic and Related Models, 1, (2008), 415--435.
\bibitem{KMT} T. K. Karper, A. Mellet, and K. Trivisa, Existence of weak solutions to kinetic flocking models, SIAM J. Math. Anal., 45, (2013), 215--243.
\bibitem{KMT2} T. K. Karper, A. Mellet, and K. Trivisa, On strong local alignment in the kinetic Cucker-Smale model, Hyperbolic Conservation Laws and Related Analysis with Applications Springer Proceedings in Math. Stat. 49, (2014), 227--242. 
\bibitem{KMT3} T. Karper, A. Mellet, and K. Trivisa, Hydrodynamic limit of the kinetic Cucker-Smale flocking model, Math. Mod. Meth. Appl. Sci., 25, (2015), 131--163.
\bibitem{Kawa} S. Kawashima, Systems of a hyperbolic-parabolic composite type with applications to the equations of magnetohydrodynamics, Thesis Kyoto University, 1983.
\bibitem{MT} S. Motsch and E. Tadmor, A new model for self-organized dynamics and its flocking behavior, J. Stat. Phys., 144, (2011), 923--947.
\bibitem{Vill} C. Villani, A review of mathematical topics in collisional kinetic theory In:Handbook of mathematical fluid dynamics, Vol. I, Amsterdam: North-Holland, (2002), 71--305.
\end{thebibliography}
\end{document}